\documentclass[review,onefignum,onetabnum]{siamart171218}
\makeatletter
\let\cref@override@label@type\@gobbletwo
\makeatother

\AtBeginDocument{\nolinenumbers} 

\usepackage{amssymb}
\usepackage{algorithmic}
\usepackage{booktabs}
\usepackage{multirow}
\usepackage{subcaption}
\usepackage{comment}
\usepackage{enumitem}
\newcommand{\bx}{\boldsymbol{x}}
\newcommand{\by}{\boldsymbol{y}}

\newcommand{\bG}{\boldsymbol{G}}
\newcommand{\bK}{\boldsymbol{K}}
\newcommand{\bZ}{\boldsymbol{Z}}

\newcommand{\br}{\boldsymbol{r}}

\newcommand{\bA}{\boldsymbol{A}}
\newcommand{\bb}{\boldsymbol{b}}
\newcommand{\bd}{\boldsymbol{d}}

\newcommand{\bv}{\boldsymbol{v}}
\newcommand{\bu}{\boldsymbol{u}}
\newcommand{\bg}{\boldsymbol{g}}
\newcommand{\bh}{\boldsymbol{h}}

\newcommand{\bz}{\boldsymbol{z}}
\newcommand{\bU}{\boldsymbol{U}}
\newcommand{\bV}{\boldsymbol{V}}

\newcommand{\bT}{\boldsymbol{T}}
\newcommand{\bR}{\boldsymbol{R}}
\newcommand{\bQ}{\boldsymbol{Q}}
\newcommand{\bI}{\boldsymbol{I}}
\newcommand{\bL}{\boldsymbol{L}}

\newcommand{\bB}{\boldsymbol{B}}
\newcommand{\bD}{\boldsymbol{D}}

\newcommand{\bbeta}{\boldsymbol{\eta}}

\newcommand{\norm}[1]{\left\lVert#1\right\rVert}

\newsiamthm{problem}{Problem}
\newsiamremark{remark}{Remark}

\newcommand{\RG}[1]{{\color{red}{#1}}}

\newcommand{\TheTitle}{A Warm-basis Method for Bridging Learning and Iteration: a Case Study in Fluorescence Molecular Tomography}

\title{ \TheTitle
}

\author{
  Ruchi Guo\thanks{Department of Mathematics, Sichuan University, China
    (\email{ruchiguo@scu.edu.cn}). This author is partially supported by NSFC 12571436 and NSF DMS-2309778.}
    \and 
   Jiahua Jiang\thanks{School of Mathematics, University of Birmingham, UK
    (\email{j.jiang.3@bham.ac.uk}).}
  \and
  Bangti Jin\thanks{Department of Mathematics, The Chinese University of Hong Kong, Shatin, N.T., Hong Kong
    (\email{b.jin@cuhk.edu.hk}).}
  \and
  Wuwei Ren\thanks{School of Information Science and Technology, ShanghaiTech University, China
    (\email{renww@shanghaitech.edu.cn}).}
     \and
  Jianru Zhang\thanks{School of Mathematics, University of Birmingham, UK
    (\email{jxz389@student.bham.ac.uk}).}
  }

\allowdisplaybreaks
\pdfoutput=1
\begin{document}

\renewcommand{\thefootnote}{} 
\footnotetext{\textsuperscript{\#}Alphabetical order}
\renewcommand{\thefootnote}{\arabic{footnote}} 

\maketitle
\begin{abstract}
Fluorescence Molecular Tomography (FMT) is a widely used non-invasive optical imaging technology in biomedical research. 
It usually faces significant accuracy challenges in depth reconstruction, and conventional
iterative methods struggle with poor $z$-resolution even with advanced regularization. 
Supervised learning approaches can improve recovery accuracy but rely on large, high-quality paired training dataset that is often impractical to acquire in practice. 
This naturally raises the question of how learning-based approaches can be effectively combined with iterative schemes to yield more accurate and stable algorithms.
In this work, we present a novel warm-basis iterative projection method (WB-IPM) and establish its theoretical underpinnings.
The method is able to achieve significantly more accurate reconstructions than the learning-based and iterative-based methods. 
In addition, it allows a weaker loss function depending solely on the directional component of the difference between ground truth and neural network output, thereby substantially reducing the training effort.
These features are justified by our error analysis as well as simulated and 
real-data experiments. 
\end{abstract}



\section{Introduction}

This work is concerned with numerically solving linear inverse problems, and the primary motivation arises from fluorescence molecular tomography (FMT)
that is a medical imaging technique known for its high sensitivity, noninvasiveness, and low cost. 
It has been widely used in various applications, including drug development, preclinical diagnosis, treatment monitoring and small animal research etc \cite{ntziachristos2002fluorescence,ren2022noninvasive,ozturk2020high,Hu2020First}. 
FMT enables the three-dimensional visualization of the internal distribution of fluorescent targets (e.g., tumors and lymph nodes) excited by near-infrared light, based on measured surface-emitted fluorescence;
see \cref{fig:FMT_intro}(a) for an illustration.
However, the strong scattering of light in biological tissues, along with the restricted light penetration, results in noisy and limited boundary measurements \cite{arridge1999optical,arridge2009optical}. 
This leads to significant loss in depth-specific information, i.e., the poorer $z$-axis resolution, compared to the better reconstruction quality in the other two directions.
In this work, we develop and analyze a novel warm-basis iterative method that draws on both learning and iterative refinement to improve reconstruction accuracy.

\begin{figure}[ht!]
    \centering
    \includegraphics[width=1\linewidth]{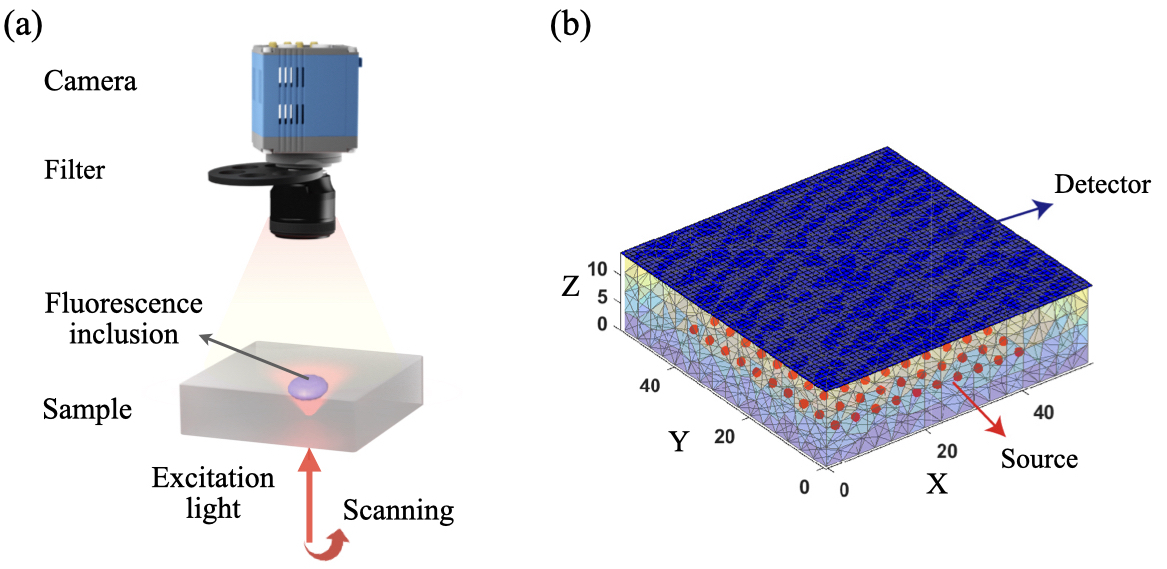}
    \caption{(a) Schematic illustration of FMT system: a laser beam scans the tissue sample from the bottom to excite fluorescent inclusions and emit fluorescence, and emitted photons propagate to the top surface and are collected by a camera. (b) Numerical simulation setup for FMT using a slab phantom: a $55\times55$ detector array (blue patches) is placed on the top surface to record photon intensity, while a $10\times10$ array of laser sources (red dots) illuminates the sample from the bottom surface.}
    \label{fig:FMT_intro}
\end{figure}

The proposed method is broadly applicable to linear inverse problems of the form:
\begin{equation}
\label{eq_LIP}
\bb = \bA \bx^* + \bbeta ,
\end{equation}
where the goal is to reconstruct the desired solution $\bx^*$, given the forward map $\bA \in \mathbb{R}^{M \times N}$ and measurement data $\bb \in \mathbb{R}^{M}$ that may include noise $\bbeta$. In the FMT application, $\bb$ and $\bx^*$ denote the surface fluorescence measurements and the internal fluorophore distribution, respectively.

The FMT problem is severely ill-posed \cite{arridge2009optical} and usually solved via regularization \cite{ItoJin:2015}.
Since the fluorescence targets (e.g., early-stage tumors and tagged biomarkers) are typically small, and sparse relative to the surrounding biological tissue \cite{hong2017near,weissleder2008imaging},
the sparsity promoting $\ell_1$ penalty is widely use \cite{ChenDonoho:1998,JinMaass:2012,JinMaassScherzer:2017}:
\begin{equation}
    \min_{\bx} \norm {\bA \bx - \bb}_2^2 + \lambda^2 \norm {\bx }_1.
    \label{eq:l1_regular}
\end{equation}
Nevertheless, solving the $\ell_1$-regularized problem is computationally demanding, 
due to its nonsmoothness \cite{rodriguez2008efficient,JinMaass:2012}. 
Popular algorithms to solve the $\ell_1$ minimization problem include Bregman iterations \cite{YinOsher:2008}, 
fast iterative soft-thresholding algorithms (FISTA) \cite{beck2009fast}, and alternating direction method of multipliers (ADMM) \cite{YangZhang:2011,kavuri2012sparsity} etc.
Krylov-type methods seek the solution within subspaces formed by repeatedly applying the matrix $\bA$ to the initial basis. In this work, we consider hybrid projection methods \cite{jiang2021hybrid,chung2019flexible, chung2017generalized} that project the problem onto low-dimensional subspaces via Golub-Kahan process.  However, conventional iterative methods struggle to recover the $z$-direction information, as the ``Depth Blur'' stems from small singular values of $\bA$ and is difficult to resolve by regularization alone. 

Recently deep learning has been widely used for inverse problems \cite{arridge2019solving,Ongie:2020}.
 Supervised learning using deep neural networks \cite{adler2017solving,adler2018learned,2020GuoJiang,2023GuoJiangLi,2023GuoCaoChen,CenJinZhou:2023} employs an end-to-end reconstruction paradigm and takes advantage of big data to recover the information lost in the physical model (e.g., depth in FMT reconstruction). These approaches have demonstrated impressive empirical results across a wide variety of applications in terms of reconstruction speed and accuracy, including FMT \cite{meng2020k,cao2022excitation,zhang2023d2}.
However, they often encounter reduced accuracy in real-world scenarios due to limitations of training dataset (see \cref{fig:exp_res}), e.g., distributional mismatch and noise contamination.

The inherent limitations in both approaches raise one fundamental question: How can learning and iterative methods be combined to achieve more accurate and stable algorithms?
A natural idea to include the prior information is ``warm start'', 
which has been explored to improve the efficiency of classical iterative methods. The network outputs are used as effective initial guesses for Newton-type solvers \cite{huang2020int,zhou2023neural}
(in order to trigger their quadratic convergence) or the conjugate gradient method \cite{ZhangLiuDong:2020}.
In enriched Krylov subspaces \cite{jiang2021hybrid, hansen2019hybrid, calvetti2003enriched}, the prior information is treated as another basis vector for solving least squares problems. Meanwhile, our 3D FMT results show that directly using the network output as an initial guess to solve (\ref{eq:l1_regular}) may even degrade performance; see \cref{fig:ws_wb_compare} and Section \ref{sec:wb-ipm} for further discussions. Thus, one must carefully design the iterative scheme such that correct information of the network prediction can be preserved and the rest can be corrected by the iteration.

In this work, we develop a new technique by decomposing the whole space into the network output and its orthogonal complement, inspired by residual analysis, rather than merely augmenting the solution space with a prior-informed basis. It essentially exploits distinct roles of the two spaces and allows the flexible hybrid projection method to efficiently search for the optimal solution within the complement. The analysis also supports the design of new training loss, and shows that the regularization parameters associated with the network output can be flexibly set to small values while still achieving performance on par with more sophisticated parameter choice rules.
Specifically, to exploit the synergy of learning and iteration, we propose a new \textit{warm-basis} iterative projection method (WB-IPM) in \cref{alg:WB-frame}. 
Our main contributions include
\begin{enumerate}
\item {\bf Learning benefits iteration.} The Attention U-Net generates a “warm basis” that, when combined with a novel alternating solver, substantially improves reconstruction accuracy compared with both learning-based and iterative methods, with notable gains in the $z$-direction.

\item {\bf Theoretical guarantees.} We establish that the performance of WB-IPM depends only on the angle between the true solution and the network output, up to noise and regularization terms. 

\item {\bf Iteration benefits learning.} The analysis motivates a weaker angle-based loss, greatly improving training efficiency while preserving the reconstruction quality of iterative refinement (see \cref{fig:convergence_history}). 

\end{enumerate} 
In sum, WB-IPM is tolerant to inaccuracies in the learned warm basis, can provide stable refinement even from imperfect network outputs.
We illustrate these features through simulated and real-world data, cf. \cref{fig:FMT_intro} and \cref{fig:exp_res}.

The paper is organized as follows. In Section \ref{sec:method}, we review the majorization-minimization approach and introduce the attention U-Net for generating warm basis. 
In Section \ref{sec:wb-ipm}, we present WB-IPM, including space decomposition, warm-basis alternating solver and AFGK iterative method. The theoretical analysis is given in Section \ref{sec:assessment}, followed by simulated and experimental results in Section \ref{sec:results}, and conclusions in Section \ref{sec:conclusion}.

\section{Preliminary} \label{sec:method}
In this section, we first describe the general iterative scheme based on an majorization-minimization (MM) and introduce the attention-type network to generate the warm basis. 

\subsection{MM approach}
\label{sec:irn}
A range of methods have been developed to solve the $\ell_1$-regularized problem, including iterative shrinkage algorithms and iterative reweighted norms \cite{gorodnitsky1992new,  rodriguez2008efficient, beck2009fast, daubechies2010iteratively}. 
We employ the MM approach, which reformulates (\ref{eq:l1_regular}) into a sequence of reweighted least-squares problems \cite{lange2016mm}. Throughout we  fix $\lambda$. For a given $\varepsilon > 0$, we approximate the absolute value $|x|$ by $ \varphi_{\varepsilon}(x) = \sqrt{x^2 + \varepsilon},$
and accordingly, approximate the $\ell_1$ norm by $\|\bx\|_1 \approx \sum_{j=1}^{N} \varphi_{\varepsilon}(x_j),$
where $x_j$ denotes the $j$th element of $\bx$. The smoothed objective is given by
\begin{equation}
    f_{\varepsilon}(\bx) = \|\bA \bx - \bb\|_2^2 + \lambda^2 \sum_{j=1}^{N} \varphi_{\varepsilon}(x_j).
    \label{prob:mmprob1}
\end{equation}
Let $\bx^{(k)}$ be the iterate at the $k$th step of the MM approach. Then the following majorization relationship holds \cite[(1.5)]{lange2016mm}
\begin{equation}
    \varphi_{\varepsilon}(x) = \sqrt{x^2 + \varepsilon} \leq \sqrt{(x^{(k)})^2 + \varepsilon} + \frac{1}{2\sqrt{(x^{(k)})^2 + \varepsilon}} (x^2 - (x^{(k)})^2) =: \psi_{\varepsilon}(x \mid x^{(k)}),
\end{equation}
i.e., the quadratic function $\psi_{\varepsilon}(x \mid x^{(k)})$ majorizes for the function $\varphi_{\varepsilon}(x)$ at $\bx^{(k)}$. Then we define a surrogate function to (\ref{prob:mmprob1}) by
\begin{equation}
    g_{\epsilon}(\bx \mid \bx^{(k)}) = \|\bA \bx - \bb\|_2^2 + \lambda^2 \sum_{j=1}^{N} \psi_{\varepsilon}(x_j \mid x_j^{(k)}).
\end{equation}
It can be readily verified that
\begin{equation}
f_{\varepsilon}(\bx^{(k)}) = g_{\varepsilon}(\bx^{(k)} \mid \bx^{(k)}) 
~~~\text{and} ~~~
f_{\varepsilon}(\bx) \leq g_{\varepsilon}(\bx \mid \bx^{(k)}) \quad \forall \, \bx \in \mathbb{R}^N.
\label{surrogate_properities}
\end{equation}
That is, the surrogate $g_{\varepsilon}(\bx \mid \bx^{(k)})$ coincides with $f_{\varepsilon}(\bx)$ at the current iterate $\bx^{(k)}$ and upper bounds it for any $\bx$. Thus, by taking the next iterate $\bx^{(k+1)}$ such that the surrogate $g_{\varepsilon}$ decreases, we ensure that the objective $f_{\varepsilon}$ also decreases:
\begin{align}
    f_{\epsilon}(\bx^{(k+1)}) \leq g_{\varepsilon}(\bx^{(k+1)} \mid \bx^{(k)}) \leq g_{\varepsilon}(\bx^{(k)} \mid \bx^{(k)}) = f_{\epsilon}(\bx^{(k)}).
\end{align}
These inequalities follow directly from  \eqref{surrogate_properities}.
Note that it is unnecessary to fully minimize the surrogate at each iteration. The MM algorithm for solving \eqref{eq:l1_regular} reads: given an initial guess $\bx^{(0)}$, we solve a sequence of reweighted least-squares problems
\begin{align}
    \bx^{(k+1)} = \arg \min_{\bx \in \mathbb{R}^N} g_{\varepsilon}(\bx,  \mid \bx^{(k)}) 
    = \arg \min_{\bx \in \mathbb{R}^N} \|\bA \bx   - \bb\|_2^2 + \lambda^2 \|\bL(\bx^{(k)})\bx \|_2^2,
    \label{prob:mmprob2}
\end{align}
with the diagonal matrix $\bL(\bx) = \operatorname{diag}([2\sqrt{x_i^2 + \varepsilon}]^{-1/2})_{i=1}^N$.

The convergence of the MM approach has been rigorously established (see, e.g., \cite{huang2017some}). However, minimizing the surrogate function $g_\varepsilon(\bx)$ for FMT requires solving problem \eqref{prob:mmprob2} with $N$ unknowns at each iteration. For small-scale problems, the exact solution can be obtained directly by solving normal equations; but for large-scale problems (e.g. FMT reconstruction), an iterative method is typically used, leading to computationally intensive inner-outer iterations \cite{rodriguez2008efficient}. 

To accelerate the convergence, inspired by recent advances (e.g., \cite{chung2019flexible, jiang2021hybrid}), we propose an approximation to \eqref{prob:mmprob2}. Specifically, given the current search subspace $\mathcal{V}_k  \subset \mathbb{R}^N$, we define a transformed subspace as $\mathcal{L}_k(\mathcal{V}_k) = \operatorname{span}\{ \bL_1^{-1} \bv_1, \dots, \bL_k^{-1} \bv_k \}$, where $\{\bv_j\}_
{j=1}^k$ are basis vectors of $\mathcal{V}_k$, and $\{\bL_j\}_{j=1}^k$ are preconditioning matrices defined by $\bL_j = \bL(\bx^{(j)})$ for $j \geq 2$ and $\bL_1 = \bI$. $\mathcal{L}_k$ can thus be interpreted as a variable preconditioner applied to $\mathcal{V}_k$. We then seek an approximate solution to  \eqref{prob:mmprob2} within the transformed subspace by solving 
\begin{equation}
    \label{prob:mmprob2_eq2}
    \bx^{(k+1)}  
    = \arg \min_{\bx \in \mathcal{L}_k(\mathcal{V}_k)} \|\bA \bx   - \bb\|_2^2 + \lambda^2 \| \bL_k\bx \|_2^2.
\end{equation}
Note that the search space $\mathcal{V}_k$ expands progressively and will span the full space $\mathbb{R}^N$ after $N$ iterations, though reaching the full space is often unnecessary in practice. 

\begin{figure}[h!] 
\centering
   \includegraphics[width = \textwidth]{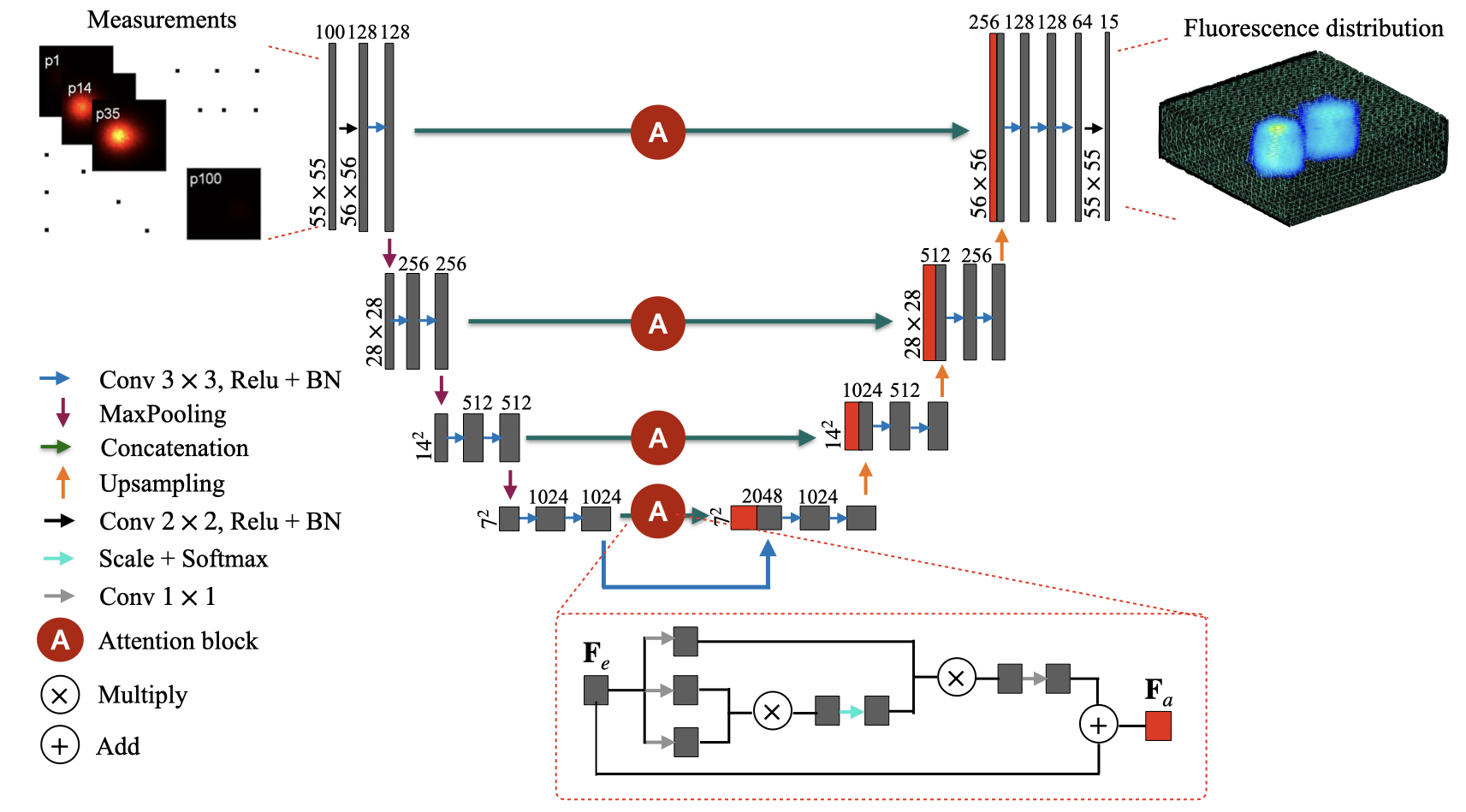}
   \caption{Overview of the Attention U-Net for generating warm basis. The network adopts an encoder-decoder structure with attention-enhanced skip connections. The encoder extracts multiscale features through convolutional layers with ReLU, batch normalization, and max-pooling, while the decoder reconstructs the fluorescence distribution by integrating encoder features via self-attention blocks.}
    \label{fig:nn}
\end{figure}

\subsection{Attention U-Net for producing warm basis}
\label{sec:aunet}
CNN-based architectures (e.g., U-Net) are widely used in imaging \cite{ronneberger2015u,litjens2017survey,cciccek20163d}. Standard encoder–decoder designs with skip connections may propagate low-level noise from early layers, which is critical in FMT due to limited and noisy boundary data \cite{li2022attention,nodirov2022attention}. In addition, FMT requires modeling long-range dependencies between surface measurements and internal fluorescence. To address this, self-attention blocks are embedded prior to feature concatenation, capturing global dependencies, and suppressing irrelevant patterns. This yields more robust reconstructions by retaining diagnostically meaningful features. Self-attention has shown success in natural image analysis \cite{vaswani2017attention,dosovitskiy2020image}, medical imaging \cite{chen2021transunet,nodirov2022attention}, and multimodal tasks \cite{radford2021learning}. Motivated by this, we adopt the Attention U-Net architecture in \cref{fig:nn} to design our neural network framework for prediction.  

Given encoder features $\mathbf{F}_e$, the Attention block reads as:
\begin{equation}
    \mathbf{F}_a = \mathbf{F}_e + \mathbf{W}^o * \left( \operatorname{softmax}\!\Biggl(\frac{(\mathbf{W}^q * \mathbf{F}_e)(\mathbf{W}^k * \mathbf{F}_e)^\top}{\sqrt{d_k}}\Biggr)(\mathbf{W}^v * \mathbf{F}_e) \right),
\end{equation}
where $\mathbf{W}^q$, $\mathbf{W}^k$ and $\mathbf{W}^v$ are $1 \times 1$ convolutions that project $\mathbf{F}_e$ into query, key, and value tensors, respectively.  These three learnable weight projections introduce global, data-adaptive interactions across all spatial locations, allowing the network to emphasize relevant features, especially when certain signals are ambiguous or corrupted (e.g. depth information affected by limited light penetration and scattering).  $d_k$ is the dimension of query/key projections, $\operatorname{softmax}(\cdot)$ operates along the key dimension and $\mathbf{W}^o$ is the output projection convolution. The architecture of the proposed Attention U-Net is depicted  in the zoom-in region of~\cref{fig:nn}.

\subsection{Network prediction and iteration}
Now we show that a-priori information and iterative schemes must mutually benefit each other to achieve high accuracy.

First, simply using the neural network prediction as the initial guess does not always guarantee accuracy improvement; see \cref{fig:ws_wb_compare} for illustration. 
The main reason is that high-frequency image details, associated with small singular values of $\bA$, are highly sensitive to noise and regularization, and errors in these modes may be amplified in iteration and cannot be adequately controlled by regularization alone.
This also results in poor FMT performance along the $z$-axis; 
see Table \ref{tab:error_results}.
One subtle reason for this issue concerns the discrepancy between the data-guided loss in training and the iteration objective, which strongly relies on in-distribution data. 
In \cref{fig:ws_wb_compare}, the experimental data deviate from that of the training data, posing a challenge for network generalization, and the warm-start methods exhibit pronounced errors.

\begin{figure}[h]
    \centering
    \includegraphics[width=1\linewidth]{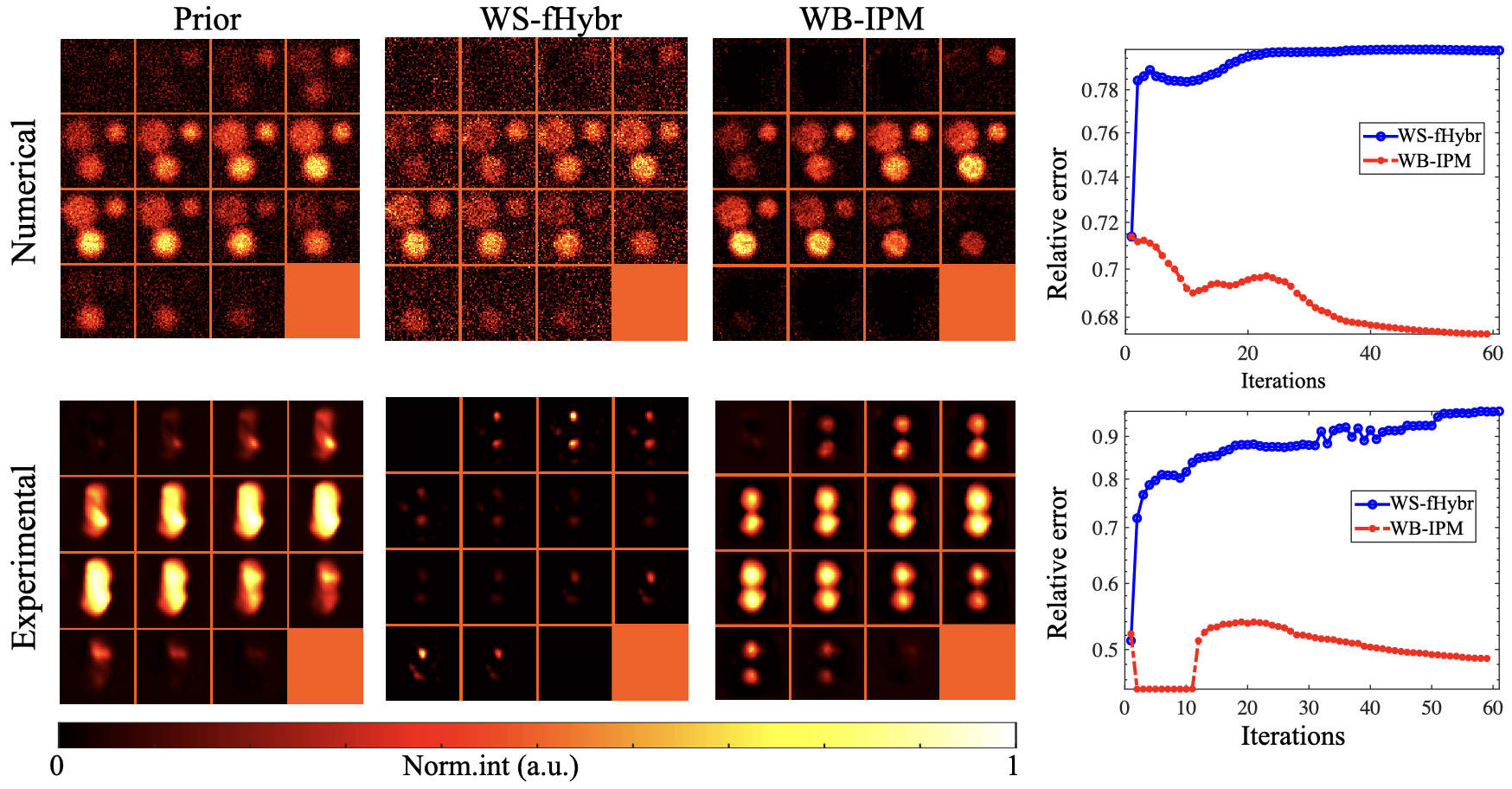}
    \caption{Comparison of a warm start method with fHybr iteration and WB-IPM still with fHybr iteration for numerical (top) and experimental (bottom) cases. 
    Warm start risks degrading prior prediction, while WB-IPM ensures robustness and accuracy. 
    }
    \label{fig:ws_wb_compare}
\end{figure}

Meanwhile, the a-priori information obtained by the neural network with data may just be compensated by iterative schemes. To show this, consider an experiment in which the proposed WB-IPM is applied to an initial basis produced by the fHybr method rather than the network.
It shows that the information hidden in the warm basis by the network cannot be recovered by repeatedly applying the iterative schemes, indicating that the network provides additional information beyond the subspaces generated by fHybr.
\begin{figure}[h]
    \centering
    \includegraphics[width=0.6\linewidth]{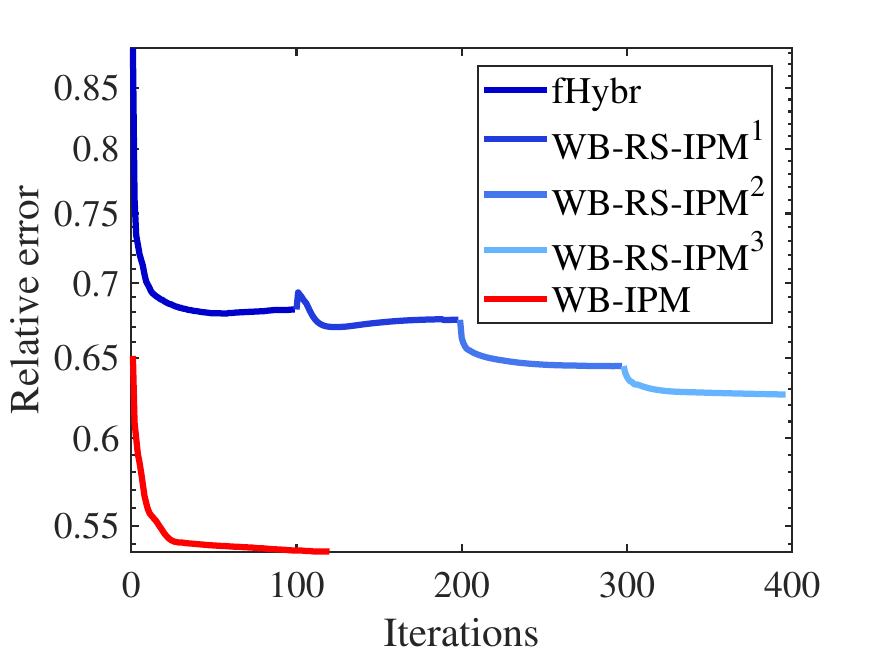}
    \caption{Relative error under different initialization strategies.
WB-RS-IPM denotes Warm-Basis Restarted IPM: $\mathrm{WB\text{-}RS\text{-}IPM}^{k}$
runs IPM in $k$ stages, where stage $1$ is initialized by fHybr, and later stages warm-started from previous reconstructions.  Restarts
yield modest accuracy gains, whereas NN-initialized WB-IPM achieves 
the best accuracy.
 }
    \label{fig:stage_changing}
\end{figure}

\section{A Warm-basis iterative projection method} \label{sec:wb-ipm}
In this section, we present the warm-basis iterative projection method (WB-IPM) that can exploit the data-driven prior information.

\subsection{Space decomposition}
\label{sec:alternating_solver} 
The key is to employ a space decomposition with an alternating solver.
Suppose that $\bA, \bb$ and $\bx_{\text{nn}}$ are given. Let $\widehat{\bx}_{\text{nn}} = \bx_{\text{nn}}/\norm{\bx_{\text{nn}}}$ be a normalized NN-initialized basis. To include $\widehat{\bx}_{\text{nn}}$ as a basis in the solution space, we decompose the solution $\boldsymbol{x}$ as
\begin{equation}
    \bx = c \widehat{\bx}_{\text{nn}} + \bz, ~~~~~ \bz \perp \widehat\bx_{\text{nn}},
\end{equation}
and project problem \eqref{prob:mmprob2} onto the subspace spanned by $\widehat{\bx}_{\text{nn}}$ and the iteratively generated space $\mathcal{Z}_k \perp  \widehat{\bx}_{\text{nn}}$. 
Without loss of generality, we assume $\bA \widehat{\bx}_{\text{nn}} \notin \mathcal{R}(\bb)$, where $\mathcal{R}(\cdot)$ denotes the range;
otherwise $\mathcal{R}(\widehat{\bx}_{\text{nn}})$ already contains the solution, and no further augmentation is needed.
Let $\by = \bA\widehat{\bx}_{\text{nn}}/\gamma$ with $\gamma = \norm{\bA\widehat{\bx}_{\text{nn}}}_2$. Then,
\begin{align}
\label{eq_Axz}
\norm{\bA\bx - \bb}_2^2
&=\norm{\bA(c \widehat{\bx}_{\text{nn}}  + \bz) - \bb}_2^2 
= \norm{ (c \gamma \by  + \bA\bz)  - \bb}_2^2 \notag \\
&= \norm{ (\bI - \by\by^\top) \bA\bz  - \bb + \by\by^\top (c \gamma \by  +  \bA\bz ) }_2^2 \notag\\
&= \norm{ (\bI - \by\by^\top) (\bA{\bz} - \bb) }_2^2 + \norm{ (c \gamma  +  \by^\top\bA\bz)-\by^\top \bb }_2^2 \notag \\
&= \norm{ \widetilde{\bA}{\bz} - \widetilde{\bb} }_2^2 + \norm{ \gamma c+\by^\top \bA{\bz}  - \by^\top \bb }_2^2, \notag
\end{align}
where $\widetilde{\bA} = (\bI - \by\by^\top )\bA$ and $\widetilde{\bb} = (\bI- \by\by^{\top})\bb$. Motivated by the analysis, we propose a warm-basis alternating solver to approximate the solution of \eqref{prob:mmprob2} in \cref{alg:WB-frame}. Note that \eqref{eq_ck} is a one-dimensional problem, but \eqref{eq_zk} involves a varying weighting matrix and is more expensive to solve. To maintain computational efficiency, we employ the AFGK iterative method in Subsection \ref{sec:AFGK}, which includes the construction of subspaces $\mathcal{Z}_k$, the explicit solution of \eqref{eq_ck}, and a strategy to select $\lambda_k$ and $\alpha_k$.

\begin{algorithm}[hbt!]
\caption{Warm-basis alternating solver}
\begin{algorithmic}[ht!]
\label{alg:WB-frame}
\REQUIRE $\bA$, $\bb$, the data-driven $\widehat{\bx}_{\text{nn}}$,
 the random initial guesses $\bz_0$ and $c_0$.
\STATE Obtain $\widetilde{\bA}$, $\widetilde{\bb}$ 
\WHILE {the stopped criteria not satisfied}
\STATE Generate data-driven subspace $\mathcal{V}_k$ orthogonal to $\widehat{\bx}_{\text{nn}}$ by the method in Subsection \ref{sec:AFGK}. 
\STATE Generate a linear mapping $\mathcal{L}_k$ by the matrices $\{\bL_j = \bL(\bz^{(j)})\}_{j=1}^k$.
\STATE Generate the searching space $\mathcal{Z}_k = \mathcal{L}_k(\mathcal{V}_k) \cap \{\widehat{\bx}_{\text{nn}}\}^{\perp}$ and compute
\begin{align}
 & \bz_{k+1} 
    = \arg \min_{\bz \in \mathcal{Z}_k } \| \widetilde{\bA} \bz   - \widetilde{\bb}\|_2^2 + \lambda_k^2 \|\bL_k\bz \|_2^2, \label{eq_zk} \\
 & c_{k+1} = \arg \min_{c \in \mathbb{R}}  \norm{ \gamma c+\by^\top \bA{\bz_{k+1}}  - \by^\top \bb }_2^2 + \alpha_k^2 c^2. \label{eq_ck}
\end{align}
\ENDWHILE
\STATE Compute the final solution $\bx = c_{k+1} \widehat{\bx}_{\text{nn}} + \bz_{k+1}$.
\end{algorithmic}
\end{algorithm}

\subsection{AFGK iterative method\label{sec:AFGK}}
We exploit aspects of both flexible \cite{chung2019flexible} and recycling Golub-Kahan projection methods \cite{jiang2021hybrid} to develop an augmented flexible Golub-Kahan (AFGK) projection method with two main components.  First, we generate a single basis in $\mathcal{Z}_k$, using a flexible preconditioning framework integrated with an orthogonality constraint with respect to $\bx_{\text{nn}}$. Second, we compute an approximate solution by solving a regularized optimization problem in the projected subspace, where the regularization parameter $\lambda_k$ is estimated automatically. 

\emph{AFGK process.}
Given $\widetilde{\bA}, \widetilde{\bb}$, a sequence of varying preconditioners $\{\bL_j\}_{j=1}^{k}$, and warm basis $\widehat{\bx}_{\text{nn}}$, we initialize the iterations with a vector $\bu_1 = \widetilde{\bb} / \beta$, where $\beta = \|\widetilde{\bb}\|_2$. The $k$th iteration of AFGK method generates vectors $\bz_{k}$, $\bv_{k}$, and $\bu_{k+1}$ such that 
\begin{align}
\widetilde{\bA} \bZ_k &= \bU_{k+1} \bG_k, \label{eq:U_recycling} \\
\widetilde{\bA}^\top \bU_{k+1} &= \bV_{k+1} \bT_{k+1} \label{eq:V_recycling}, 
\end{align}
where $\bZ_k = (\mathbf{I} - \widehat{\bx}_{\text{nn}} \widehat{\bx}_{\text{nn}}^\top) \begin{bmatrix}\mathbf{L}_1^{-1}\bv_1 &\ldots&\mathbf{L}_k^{-1}\bv_k\end{bmatrix} \in \mathbb{R}^{N \times k}$, $\bU_{k+1} = \begin{bmatrix}\bu_1 & \ldots & \bu_{k+1}\end{bmatrix} \in \mathbb{R}^{M \times (k+1)}$ has orthonormal columns, $\bG_k \in \mathbb{R}^{(k+1) \times k}$ is upper Hessenberg and $\bT_{k+1} \in \mathbb{R}^{(k+1) \times (k+1)}$ is upper triangular.
We verify several orthogonality conditions involving $\widehat{\bx}_{\text{nn}}$ as follows. 
Note that $\widehat{\bx}_{\text{nn}} \perp \bZ_k$. Let $\mathcal{V}_k = \mathcal{R}(\bV_k)$. Then the search space in \eqref{eq_ck} is given by $\mathcal{Z}_k = \mathcal{L}_k(\mathcal{V}_k) = \mathcal{R}(\bZ_k)$. In exact arithmetic, the solution spaces of the two subproblems \eqref{eq_zk} and \eqref{eq_ck} are mutually orthogonal without the need for explicit orthogonalization, i.e.,
$\widehat{\bx}_{\text{nn}} \perp \mathcal{Z}_k$.

\vspace{0.1in}

\emph{Solving the least squares problem.}
Next, we seek an approximate solution to the least squares problem \eqref{eq_zk} in $\mathcal{Z}_k$. To determine the coefficients $\bd_k$, we plug the AFGK relations \eqref{eq:U_recycling} and \eqref{eq:V_recycling} into \eqref{eq_zk} and obtain 
\begin{align}
    \bd_k 
    &= \arg \min_{\bd \in \mathbb{R}^k} \| \widetilde{\bA} \bZ_k\bd   - \widetilde{\bb}\|_2^2 + \lambda_k^2 \|\bZ_k\bd \|_2^2 \\
    & = \arg \min_{\bd \in \mathbb{R}^{k}} \norm{ \bG_k\bd - \beta \mathbf{e}_1}_2^2 + \lambda_k^2 \norm{ \bR_{Z,k}\bd }_2^2,  \label{prob_proj_d}
\end{align}
where a thin QR factorization is performed on $\bZ_k  = \bQ_{Z,k}\bR_{Z,k}$ with $\bQ_{Z,k} \in \mathbb{R}^{N\times k}$ and $\bR_{Z,k} \in \mathbb{R}^{k\times k}$. The details of the QR factorization are given in Appdenix \ref{sec:eff-QR}. To determine $c_k$, we substitute $\bd_k$, \eqref{eq:U_recycling} and \eqref{eq:V_recycling} into \eqref{eq_ck}, and get
\begin{align}
    c_{k} &= \arg \min_{c \in \mathbb{R}}  \norm{ \gamma c+\by^\top \bA\bZ_k\bd_k  - \by^\top \bb }_2^2 + \alpha_k^2 c^2 \label{eq:ck_lam_1} 
     = \frac{\gamma (\by^\top\bb-\by^\top \bA \bZ_k \bd_k ) }{\gamma^2+\alpha_k^2}. 
\end{align}
The regularization parameters $\lambda_k$ and $\alpha_k$  can be  efficiently and automatically estimated by applying  standard parameter selection techniques, e.g., the weighted generalized cross-validation (WGCV) method \cite{chung2008weighted}, to the projected problems  \eqref{prob_proj_d} and \eqref{eq:ck_lam_1}, respectively. The AFGK method is summarized in \cref{alg:Augmented_FGK}. By Theorem 3.2 in \cite{chung2019flexible}, the solution subspace generated by AFGK coincides with 
\begin{align*}
    \mathcal{R}(\begin{bmatrix}\widehat{\bx}_{\text{nn}}, \bZ_k\end{bmatrix}) &= \text{Span} \left\{\widehat{\bx}_{\text{nn}},  \prod_{i=2}^k \bK_i \bL_1^{-1} \bA^\top (\bI - \by\by^\top )\bb\ \right\}
\end{align*}
with $\bK_i = \bL_i^{-1} \bA^\top (\bI - \by\by^\top )\bA$, as $ (\bI - \by\by^\top )^\top = \bI - \by\by^\top$ and $(\bI - \by\by^\top )^2 = \bI - \by\by^\top$.

\begin{algorithm}[h!]
\caption{Augmented flexible Golub-Kahan (AFGK) Process \RG{}}
\begin{algorithmic}[ht!]
\label{alg:Augmented_FGK}
    \STATE Initialize $\bu_1 = \widetilde{\bb}/ \beta$, where $\beta = \norm{\widetilde{\bb}}_2$
    \FOR{$i=1, \ldots, k$}
        \STATE Compute $ \bh = \bu_i - \by(\by^\top \bu_i), \bh = \bA^\top \bh$, 
        \STATE $t_{ji} = \bh^\top \bv_j$ for $j = 1, \ldots, i - 1$
        \STATE Set $\bh = \bh - \sum_{j=1}^{i-1} t_{ji} \mathbf{v}_j$, compute $t_{ii} = \norm{\bh}_2$ and take $\mathbf{v}_i = \bh / t_{ii}$
        \STATE Compute $\mathbf{z}_i = \mathbf{L}_i^{-1} \mathbf{v}_i$, $\mathbf{z}_i = \mathbf{z}_i - \widehat{\bx}_{\text{nn}}(\widehat{\bx}_{\text{nn}}^\top \mathbf{z}_i)$
        \STATE Set $\bh = \bA \mathbf{z}_i$, $\bh = \bh - \by(\by^\top \bh)$
        \STATE $g_{ji} = \bh^\top \mathbf{u}_j$ for $j = 1, \ldots, i$ and set $\bh = \bh - \sum_{j=1}^{i} g_{ji} \mathbf{v}_j$
        \STATE Compute $g_{i+1,i} = \norm{\bh}_2$ and take $\mathbf{u}_{i+1} = \bh / g_{i+1,i}$
    \ENDFOR
\end{algorithmic}
\end{algorithm}

\section{Analysis of the WB-IPM\label{sec:assessment}}
In this section, we analyze the residuals and errors of the solutions produced by the WB-IPM. 
The analysis motivates the design and highlights the benefits of the WB-IPM.
Below we often use the space $\ker(\widetilde{\bA})$
which can be represented as
\begin{equation}
\label{kerA}
\mathcal{K}:=\ker(\widetilde{\bA}) = \text{Span}\{ \bx_{nn} \} + \ker(\bA).
\end{equation}
For any vector $\mathbf{v}$,
$\|\mathbf{v}\|_\infty = \max_{i} |v_i|,$
and for any matrix $\mathbf{M}$, $\sigma_{\max}(\mathbf{M})$ and $\sigma_{\min}(\mathbf{M})$ denote the maximum and minimum singular values of $\mathbf{M}$, respectively. Let $\widetilde{\bB} = \widetilde{\bA}^\top\widetilde{\bA}$, 
$\bD = \bL_z^\top\bL_z$
and $\bD_{\lambda} = \lambda^2 \bL_z^\top\bL_z$ where  the regularization parameter $\lambda$ and invertible preconditioner $\bL_z$ are fixed in the analysis below. 
Note that $ \widetilde{\bB} $ is positive semidefinite and $ \bD_{\lambda} $ is positive definite, both symmetric, 
and $\widetilde{\bB} +\bD_{\lambda} $ is symmetric and positive definite. 
The solution $\bz_w$ to \eqref{eq_zk} after $N$ iterations is given by
\begin{equation}
\label{z_solu1}
\bz_w = ( I - \hat{\bx}_{\text{nn}}\hat{\bx}^\top_{\text{nn}}  ) (\widetilde{\bB} + \bD_{\lambda})^{-1} \widetilde{\bA}^\top\widetilde{\bb},
\end{equation}
and the solution $c_w$ to \eqref{eq_ck} is given by
\begin{equation}
\label{z_solu2}
c_w =  \frac{\gamma \by^\top ( \bb - \bA \bz_w )}{\gamma^2+\alpha^2}.
\end{equation}
Then, the WB-IPM solution can be written as
\begin{equation}
\label{z_solu3}
\bx_w = \bz_w + c_w \hat{\bx}_{\text{nn}}.
\end{equation}
Meanwhile, consider the orthogonal decomposition of the true solution $\bx^*$: 
\begin{equation}
\label{thm:error_bounds_eq1}
\bx^* = \bz^* + c^*\widehat{\bx}_{\text{nn}}, ~~~ \text{with} ~ c^*= (\bx^*)^\top\widehat{\bx}_{\text{nn}},~~ \bz^* \bot \; \widehat{\bx}_{\text{nn}}
\end{equation}
Then $c^*\widehat{\bx}_{\text{nn}} \in \mathcal{K}$. Moreover, define the quantity: 
\begin{equation}
\label{Gamma}
\Gamma(\bv;\mathcal{S}) = \| \bv_\bot \|^2_{\bD}/\| \bv \|^2_{\bD},
\end{equation}
with $\bv_\bot$ being the $\bD$-projection of $\bv$ onto the space $\mathcal{S}$, to measure the error.
Note that $\Gamma(\bv;\mathcal{S})  \le 1$,
and it is $1$ only if $\bv\in\mathcal{S}$.

\begin{lemma}
\label{lem_DB}
Given any vector $\bv$, there holds
\begin{equation}
\label{lem_DB_eq1}
 \| (\widetilde{\bB} + \bD_{\lambda})^{-1} \bD_{\lambda} \bv \|_{2} 
 \le  \frac{1}{\sigma_{\min}(\bL_z)} 
 \left( \Gamma(\bv;\mathcal{K}) +  \frac{\lambda^4(1- \Gamma(\bv;\mathcal{K}))}{(\lambda^2+\theta_+)^2} \right)^{\frac{1}{2}} \| \bv \|_{\bD},
\end{equation}
where $\theta_+$ is the smallest non-zero eigenvalue of $\bD^{-1}\widetilde{\bB}$. 
\end{lemma}

\begin{proof}
We first estimate the eigenvalues of 
$(\widetilde{\bB} + \bD_{\lambda})^{-1} \bD_{\lambda}$, denoted by $\mu_1 \ge \dots \ge \mu_N$. Let $\bg_k$ be the eigenvector corresponding to $\mu_k$. 
Let $\theta_1 \ge \dots \ge\theta_N \ge 0$ be the eigenvalues of the matrix $\bD^{-1}\widetilde{\bB}$.
Then, there holds
\begin{align}
(\widetilde{\bB} + \bD_{\lambda})^{-1}\bD_{\lambda}\bg_k = \mu_k \bg_k \label{eq:error_vc_wb-ipm_4}
~~~
\Longleftrightarrow
~~~
\mu_k (\bD_{\lambda}^{-1}\widetilde{\bB} + \bI) \bg_k= \bg_k,
\end{align}
which implies
\begin{equation}
\label{eq:error_vc_wb-ipm_41}
\mu_k  = \frac{\lambda^2}{\lambda^2 + \theta_{N-k+1}}.
\end{equation}
Since $\widetilde{\bB}$ is singular, we have $\theta_{N - N_0} > \theta_{N - N_0+1} = \dots =\theta_N = 0$ for some integer $N_0\ge1$. Then $\mu_1=\mu_2 = \cdots = \mu_{N_0} = 1 >  \mu_{N_0+1} \ge \cdots \mu_N > 0$ and $\{\bg_k\}_{k=1}^{N_0} \subset \ker(\widetilde{\bB})$. Since $\{\bg_k\}_{k=1}^N$ are orthogonal with respect to the $\bD_{\lambda}$-inner product and also the $\bD$-inner product due to the simple scaling.
Then, we express $\bv$ as
\begin{align}
    \bv &= \sum_{k=1}^N \gamma_k \bg_k,
\end{align}
which implies 
\begin{align}
  \|\bv \|^2_{\bD} = \sum_{k=1}^N \gamma_k^2 \|\bg_k\|^2_{\bD}.
\end{align}
Then, the $\bD_{\lambda}$-orthogonality implies
\begin{equation}
\begin{split}
\label{eq:error_vc_wb-ipm_3}
 &  \norm{(\widetilde{\bB} + \bD_{\lambda})^{-1} \bD_{\lambda} \bv }^2_2
 =    \norm{ (\widetilde{\bB} + \bD_{\lambda})^{-1} \bD_{\lambda} \sum_{k=1}^N \gamma_k \bg_k}^2_2 \\
 = & \norm{ \sum_{k=1}^N \gamma_k \mu_k \bD^{-1}\bg_k }^2_{\bD}
 \le  \sigma_{\text{max}}(\bD^{-1})\sum_{k=1}^N \gamma^2_k \mu^2_k \| \bg_k \|^2_{\bD} \\
 \le & \frac{1}{\sigma_{\text{min}}(\bD)} (\sum_{k=1}^{N_0} \gamma^2_k \| \bg_k \|^2_{\bD} +  \mu^2_{N_0+1} \sum_{k=N_0}^{N} \gamma^2_k \| \bg_k \|^2_{\bD}) \\
 = &  \frac{1}{\sigma_{\text{min}}(\bD)}\left(  \Gamma + \mu^2_{N_0+1}(1-\Gamma) \right) \| \bv \|^2_{\bD},
 \end{split}
\end{equation}
where the quantity $\Gamma$ is given as
\begin{equation}
\label{eq:error_vc_wb-ipm_5}
\begin{split}
\Gamma & = \left( \sum_{k=1}^{N_0}\gamma^2_k \| \bg_k \|^2_{\bD}  \right)/\left( \sum_{k=1}^{N}\gamma^2_k \| \bg_k \|^2_{\bD}  \right) 
 = \| \bv_\bot \|^2_{\bD}/\| \bv \|^2_{\bD}, 
\end{split}
\end{equation}
with ${\bv}_\bot$ being the $\bD$-projection of ${\bv}$ onto $\ker(\widetilde{\bB})$. 
Since $\widetilde{\bB} = \widetilde{\bA}^\top\widetilde{\bA}$ implies $\ker(\widetilde{\bB}) = \ker(\widetilde{\bA})$, 
${\bv}_\bot$ is equivalently the $\bD$-projection of ${\bv}$ onto $\ker(\widetilde{\bA})$,
i.e., $\Gamma = \Gamma(\bv;\mathcal{K})$. 
Last, by noting $\mu_{N_0+1} = \frac{\lambda^4}{(\lambda^2+\theta_{N-N_0})^2}$ from \eqref{eq:error_vc_wb-ipm_41}
with $\theta_{+} = \theta_{N-N_0} $ being the smallest non-zero eigenvalue of $\bD^{-1}\widetilde{\bB}$,
we obtain the desired result from \eqref{eq:error_vc_wb-ipm_3}.
\end{proof}

\begin{theorem}
\label{thm:error_bounds}
With the regularization parameters $\alpha>0$ and $\lambda>0$,
the following error bound holds:
\begin{align}
    \norm{\bx^* - \bx_w}_2 & \le \frac{\alpha^2 |c^*| }{\gamma^2 + \alpha^2} 
    + \mathcal{C}_1 \mathcal{C}_2 \Gamma(\bx^*, \emph{Span}(\hat{\bx}_{nn})^{\bot}) \| \bx^* \|_{\bD} 
    + \mathcal{C}_3\norm{\bbeta}_2,  \label{eq:bound_warm}
\end{align}
with 
\begin{subequations}
\begin{align}
    & \mathcal{C}_1 = \left( 1+\frac{\gamma \| \bA^\top \by \|_2}{\gamma^2 + \alpha^2} \right)/\sigma_{\min}(\bL_{\bz}) , \\
    &  \mathcal{C}_2 = \left( \Gamma(\bz^*;\mathcal{K})  +  \frac{\lambda^4}{(\lambda^2+\theta_{+})^2} (1- \Gamma(\bz^*;\mathcal{K})) \right)^{1/2}, \\
    &  \mathcal{C}_3 = \mathcal{C}_1 \frac{\sigma_{\max}(\bA)}{\lambda^2\sigma_{\min}(\bL_{\bz})} + \frac{\gamma}{\gamma^2 + \alpha^2}.
\end{align}
\end{subequations}
\end{theorem}
\begin{proof} 
Let $\widetilde{\bbeta} = (\bI - \by\by^\top)\bbeta$.
Then $\bz^*$ solves the following projected problem:
\begin{equation*}
    \widetilde{\bb} = \widetilde{\bA} (\bx^* - c^*\widehat{\bx}_{\text{nn}}) + \widetilde{\bbeta} =\widetilde{\bA}\bz^* + \widetilde{\bbeta}.
\end{equation*}
The error of the WB-IPM solution $\bx^*$ is given by
$\bx^* - \bx_{w} = \bz^* - \bz_w + (c^* - c_w)\widehat{\bx}_{\text{nn}}.$
Since $\bz^* - \bz_w \; \bot \; \widehat{\bx}_{\text{nn}}$, we have
\begin{align}
\label{xxw}
    \norm{\bx^* - \bx_{w}}_2 \le \norm{\bz^* - \bz_w}_2 + | c^* - c_w |.
\end{align}
We first estimate $c^* - c_w$. 
By \eqref{z_solu2} and \eqref{thm:error_bounds_eq1}, we have 
\begin{equation*}
    c_w   = \frac{\gamma^2 c^* + \gamma \by^\top\bA(\bz^* - \bz_w) + \gamma \by^\top \bbeta}{\gamma^2 + \alpha^2},
\end{equation*}
and thus obtain
\begin{equation}
\label{eq_ccw}
c^*- c_w  = \frac{\alpha^2}{\gamma^2 + \alpha^2}c^* - \frac{\gamma\by\top\bA}{\gamma^2 + \alpha^2}(\bz^* - \bz_w) - \frac{\gamma\by\top\bbeta}{\gamma^2 + \alpha^2}.
\end{equation}
Thus we obtain from \eqref{eq_ccw} that
\begin{equation}
\label{eq:final_err}
\begin{split}
    | c^* - c_w | &\le \frac{\alpha^2}{\gamma^2 + \alpha^2}|c^*| 
     + \frac{\gamma \| \bA^\top \by \|_2}{\gamma^2 + \alpha^2}\norm{\bz^* - \bz_w}_2
     +\frac{\gamma}{\gamma^2 + \alpha^2} \| \bbeta \|_2 . \\ 
        \end{split}
\end{equation}
To estimate $\| \bz^* - \bz_w \|$,
we use 
\begin{equation}
    \label{eq:error_vc_wb-ipm}
\begin{split}
   &  \bz^* - \bz_w 
    =  \bz^* - ( I - \hat{\bx}_{\text{nn}}\hat{\bx}^\top_{\text{nn}} ) (\widetilde{\bB} + \bD_{\lambda})^{-1}\widetilde{\bA}^\top\widetilde{\bb}  \\
    = & ( I - \hat{\bx}_{\text{nn}}\hat{\bx}^\top_{\text{nn}} )\left(\bI - (\widetilde{\bB} + \bD_{\lambda})^{-1}\widetilde{\bB}\right)\bz^* 
     - ( I - \hat{\bx}_{\text{nn}}\hat{\bx}^\top_{\text{nn}} ) (\widetilde{\bB} + \bD_{\lambda})^{-1}\widetilde{\bA}^\top \widetilde{\bbeta}  \\
    = & \underbrace{ ( I - \hat{\bx}_{\text{nn}}\hat{\bx}^\top_{\text{nn}} ) (\widetilde{\bB} + \bD_{\lambda})^{-1} \bD_{\lambda}\bz^* }_{\mathfrak{a}} 
    - \underbrace{ ( I - \hat{\bx}_{\text{nn}}\hat{\bx}^\top_{\text{nn}} ) (\widetilde{\bB} 
    +  \bD_{\lambda})^{-1}\widetilde{\bA}^\top \widetilde{\bbeta} }_{\mathfrak{b}}.
\end{split}
\end{equation}
We then estimate these two terms individually.
For $\mathfrak{a}$, by Lemma \ref{lem_DB}, 
\begin{equation}
\label{eq:est_a}
\norm{\mathfrak{a}}^2_2 \le 
\norm{(\widetilde{\bB} + \bD_{\lambda})^{-1} \bD_{\lambda}\bz^*}_2^2 
 \le \frac{ \| \bz^* \|^2_{\bD}}{\sigma_{\text{min}}(\bD)} \left( \Gamma(\bz^*;\mathcal{K})
 +  \frac{\lambda^4( 1- \Gamma(\bz^*;\mathcal{K}))}{(\lambda^2+\theta_{+})^2} \right).
\end{equation}
By Weyl's inequality,
\[
\sigma_{\min}(\widetilde{\bB} + \bD_{\lambda}) \geq \sigma_{\min}(\widetilde{\bB}) + \sigma_{\min}(\bD_{\lambda}) \ge \lambda^{2}{\sigma_{\text{min}}(\bD)}.
\]
Then
\[
\|(\widetilde{\bB} + \bD_{\lambda})^{-1}\widetilde{\bA}^{\top}\widetilde{\bbeta}\|_{2} 
\leq \frac{\sigma_{\max}(\widetilde{\bA})}{\sigma_{\min}(\widetilde{\bB} + \bD_{\lambda})} \|\widetilde{\bbeta}\|_{2} 
\leq \frac{\sigma_{\max}(\widetilde{\bA})}{\lambda^{2}\sigma_{\text{min}}(\bD)} \|\widetilde{\bbeta}\|_{2}.
\]
With $\|\widetilde{\bbeta}\| \le \| \bbeta \|$, there holds 
\begin{align}
\label{eq:est_b}
    \norm{\mathfrak{b}}_2 
\leq \frac{\sigma_{\max}(\widetilde{\bA})}{\lambda^{2}\sigma_{\text{min}}(\bD)} \|{\bbeta}\|_{2},
\end{align}
By combining \eqref{eq:est_a} and \eqref{eq:est_b} with \eqref{eq:error_vc_wb-ipm}, 
we obtain the estimate for $\norm{\bz^* - \bz_\omega}_2$:
\begin{align}
\label{eq:z_err}
    \norm{\bz^* - \bz_\omega}_2  \le &
   \frac{1}{\sigma_{\text{min}}(\bL_{\bz})} \left( \Gamma(\bz^*;\mathcal{K})
 +  \frac{\lambda^4}{(\lambda^2+\theta_{+})^2} ( 1- \Gamma(\bz^*;\mathcal{K})) \right)^{1/2} \| \bz^* \|_{\bD} \\
    &+\frac{\sigma_{\max}(\widetilde{\bA})}{\lambda^{2}\sigma_{\text{min}}(\bD)} \|{\bbeta}\|_{2}. \nonumber 
\end{align}
Note that $\| \bz^* \|_{\bD} = \Gamma(\bx^*, \text{Span}(\hat{\bx}_{nn})^{\bot}) \| \bx^* \|_{\bD}$.
Putting it into \eqref{eq:final_err} gives the estimate for $c^*-c_w$,
which finishes the proof by \eqref{xxw}. 
\end{proof}

\begin{remark}
\begin{itemize}
\item[]

\item The error bound only depends on the true solution $\bx^*$, the neural-network approximation $\widehat{\bx}_{nn}$, the noise and the regularization parameters,
with all the terms being well-controlled, 
thereby avoiding the situation in Figure \ref{fig:ws_wb_compare} that a simple warm start approach ultimately fails to improve.
\item Note that both $\mathcal{C}_2$ and $\Gamma(\bx^*, \emph{Span}(\hat{\bx}_{nn})^{\bot})$ are upper bounded by $1$. Since $\widetilde{\bA}\bz^*=\widetilde{\bA}\bx^*$ and $\widetilde{\bA}\bx^*$ is far away from $0$ in practice, it is likely $\Gamma(\bz^*;\mathcal{K})\ll 1$.
 Thus, $\mathcal{C}_2$ is close to $ \frac{\lambda^2}{\lambda^2+\theta_{+}}$. 
 Its smallness is then controlled by $\theta_+$, the smallest non-zero eigenvalue of the preconditioned matrix $\bD^{-1}\widetilde{\bB}$.
 \end{itemize}
\end{remark}

The first term in the error bound \eqref{eq:bound_warm} is small because the regularization parameter $\alpha$ is small and
$\gamma = \|\bA \widehat{\bx}_{\text{nn}}\|_2$ is large.
In fact, compared with $\lambda$,
the effective  $\alpha$ along $\bx_{\text{nn}}$ is much weaker and even negligible, since this direction is nearly aligned with the true solution.
See \cref{fig:alpha_changing} for the detailed comparison.

Since the first term in the error bound is very small, the total error may be dominated by the second term. Here, $\Gamma(\bx^*, \text{Span}(\hat{\bx}_{nn})^{\bot})$ measures how close $\bx^*$ is to the direction of $\widehat{\bx}_{nn}$. This suggests a novel loss function based on their angle:
\begin{equation}
\label{eq_angleLoss}
\mathcal{L}_{\text{angle}}(\bb;\theta)
:= 1 - (\bx^*)^\top\mathcal{N}(\bb;\theta)/(\|\bx^*\|\|\mathcal{N}(\bb;\theta)\|),
\end{equation}
for training the neural networks, rather than the usual $\ell^2$ distance function:
\begin{equation}
\label{eq_distLoss}
\mathcal{L}_{\text{dist}}(\bb;\theta):= \| \mathcal{N}(\bb;\theta) - \bx^* \|^2_2.
\end{equation}

By design, \eqref{eq_angleLoss} is weaker than \eqref{eq_distLoss}. Their training behaviors are compared in Figure~\ref{fig:convergence_history}. The angle-loss converges much faster: after 200 epochs, its loss drops to 
7.6\%($=1-92.4$\%) of the initial value, versus 
27.2\%($=1-72.8$\%) for the distance loss. 
As expected, the network trained with the weaker angle loss yields a poorer standalone prediction than the one trained with the stronger distance loss. Remarkably, the WB-IPM iterations initialized by these two predictions exhibit nearly identical performance, i.e,, they have the same final error and comparable convergence speed. This feature is highly desirable, since one can train with the weaker loss to cut training cost, without sacrificing downstream reconstruction quality after WB-IPM.

\begin{figure}[h]
    \centering
    \includegraphics[width=0.6\linewidth]{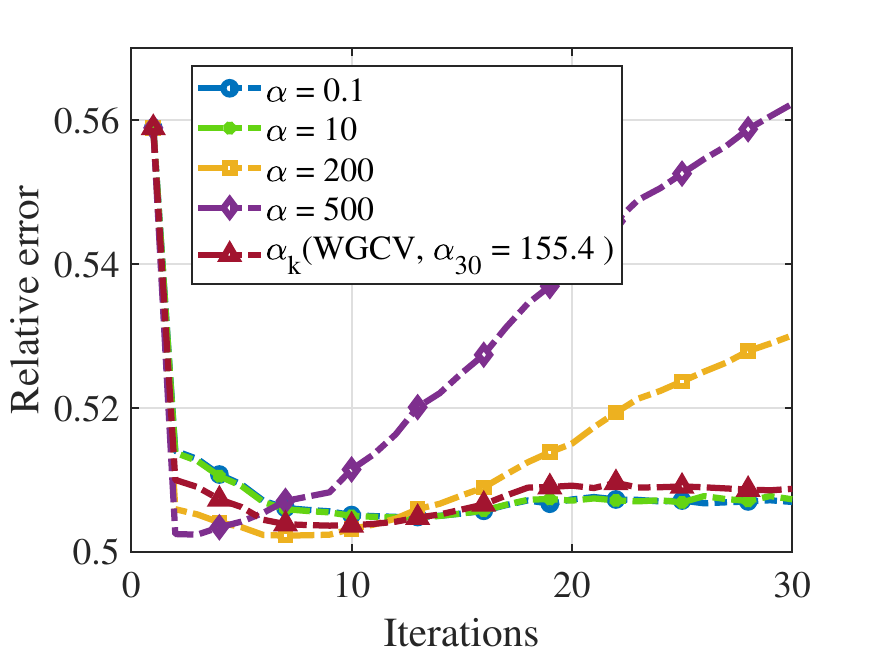}
    \caption{Relative reconstruction error of the experimental case for fixed $\alpha$ ($\alpha\in\{0.1,10,200,500\}$) and an iteration-adaptive choice $\alpha_k$ via WGCV (red triangles; $\alpha_{30}=155.4$). A small $\alpha$ achieves errors comparable to the WGCV schedule, whereas a large $\alpha$ over-regularizes and progressively degrades accuracy. }
    \label{fig:alpha_changing}
\end{figure}

\begin{figure}[h!]
  \centering
  \begin{subfigure}{0.5\linewidth}
    \centering
    \includegraphics[width=\linewidth]{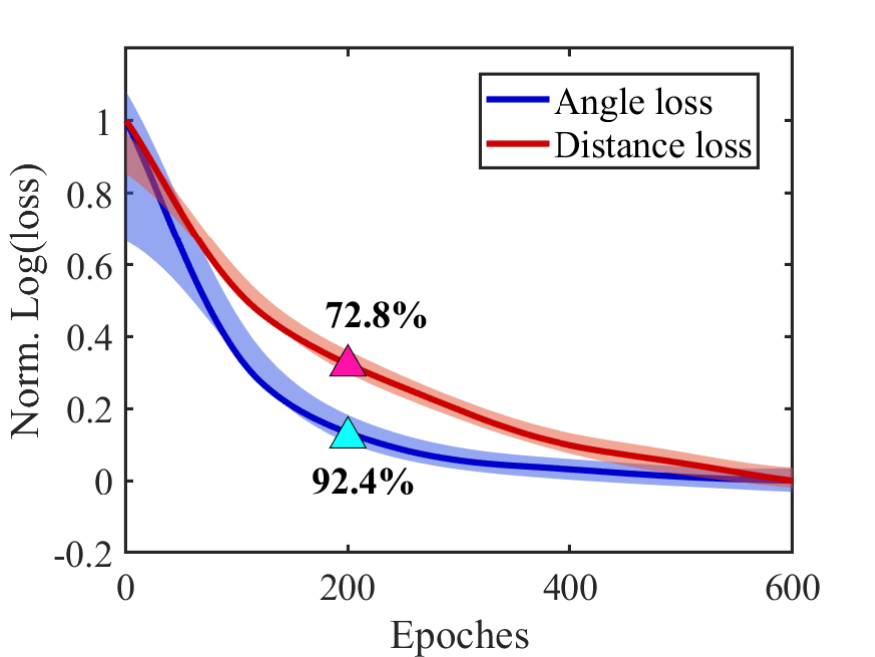}
  \end{subfigure}\hfill
  \begin{subfigure}{0.5\linewidth}
    \centering
    \includegraphics[width=\linewidth]{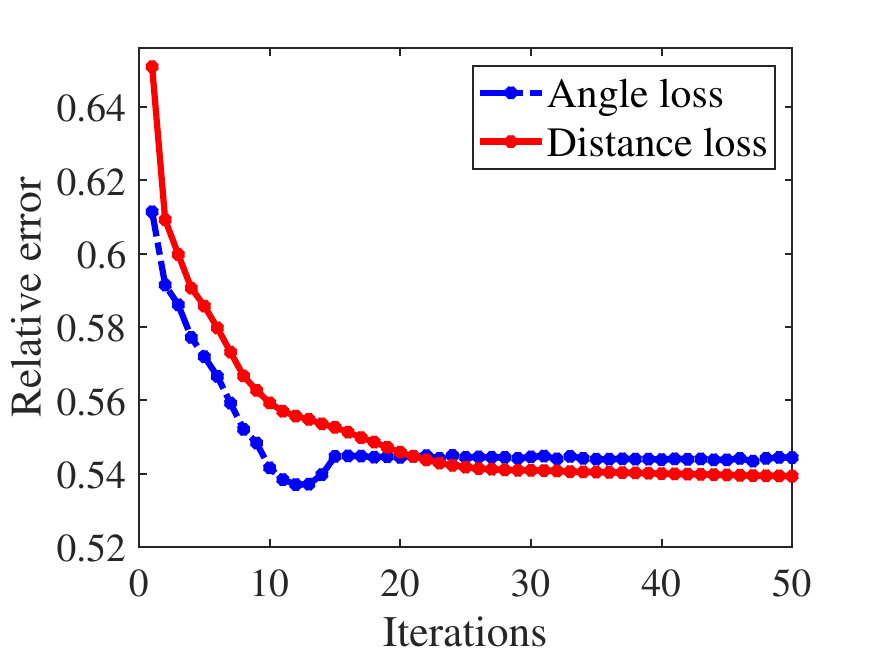}
  \end{subfigure}
  \caption{Left: the training history of the angle and distance loss shows that the convergence of angle loss is much faster.
    Right: the convergence history of WB-IPM based on predictions from the neural network based on the two different loss functions.}
  \label{fig:convergence_history}
\end{figure}

\section{Numerical experiment}\label{sec:results} In this section, we evaluate the proposed WB-IPM on both numerical (Subsection \ref{sec:numerical_results}) and experimental (Subsection \ref{sec:exp_results}) 3D FMT problems. The forward problem setup and data generation approach are detailed in Subsection \ref{sec:U-Net}. Our method can  achieve higher accuracy and robustness, particularly along $z-$ axis, compared to the original flexible hybrid projection method (denoted by \texttt{fHybr}) \cite{chung2019flexible} and the pure OpL output (Subsection \ref{sec:aunet}). 
Throughout WGCV is used to select regularization parameters, while the regularization parameter along the warm basis can be quite flexible as shown in \cref{fig:ws_wb_compare}.

\subsection{Forward problem and data generation\label{sec:U-Net}}
We first describe the forward model of FMT used to generate synthetic data for training and evaluation.
In FMT, at near-infrared wavelengths, photon transport in biological tissue is well approximated by coupled diffusion equations \cite{arridge1995photon1, arridge1995photon2, arridge1999optical} on a bounded domain $\Omega \subseteq \mathbb{R}^3$:
\begin{align}
[-\nabla\cdot \kappa^{\text{ex}}(\br)\nabla  + \mu_a^{\text{ex}}(\br)]\phi^{\text{ex}}(\br) = q^{\text{ex}}(\br), &\quad \br\in \Omega \label{eq:FMT_forward_1} \\
\phi^{\text{ex}}(\br) + 2\Gamma(\rho)\kappa^{\text{ex}}(\br)\partial_\nu\phi^{\text{ex}}(\br) = 0, &\quad \br \in \partial\Omega \label{eq:FMT_forward_2} \\
[-\nabla\cdot \kappa^{\text{em}}(\br)\nabla  + \mu_a^{\text{em}}(\br)]\phi^{\text{em}}(\br) = \eta x(\br) \phi^{\text{ex}}(\br),&\quad \br\in \Omega \label{eq:FMT_forward_3} \\
\phi^{\text{em}}(\br) + 2\Gamma(\rho)\kappa^{\text{em}}(\br)\partial_\nu \phi^{\text{em}}(\br) = 0, &\quad \br \in \partial\Omega
\label{eq:FMT_forward_4}
\end{align}
where $\nu$ is the outward normal to the boundary $\partial \Omega$, superscripts ``ex'' and ``em'' represent excitation and emission, $\phi^{\text{ex}} (   \phi^{\text{em}})$ is the photon density,
 $\mu_a^{\text{ex}}(   \mu_a^{\text{em}})$ and $\kappa^{\text{ex}} (  \kappa^{\text{em}})$  are absorption and diffusion coefficients,
$\rho$ is the light speed, $\Gamma(\rho)$ models refractive index mismatch,
$q^{\text{ex}}$ is the excitation source,
$\eta$ is the efficiency constant, and $x(\br)$ is the fluorophore distribution to be reconstructed.

The forward map is constructed from \eqref{eq:FMT_forward_1}–\eqref{eq:FMT_forward_4} and discretized using FEM \cite{schweiger1997finite,schweiger2014toast++,ren2019smart} to compute excitation and emission fields for various fluorescence distributions $\bx$. 
The training dataset consists of pairs $(\bI,\mathbf{P}_d\Phi^{\mathrm{em}})$, where $\bI$ represents $1$–$3$ random inclusion and $\mathbf{P}_d$ projects the emission field to detector measurements.
The simulated phantom is modeled as a $54 \times 54 \times 14 \text{ mm}^3$ slab, illuminated by a $10 \times 10$ laser grid and measured on a $55 \times 55$ detector array (see \cref{fig:FMT_intro}(b)). 

\subsection{Simulation results\label{sec:numerical_results}} 
We present three simulated cases, visualized as $z$-axis slices on a $55\times55\times15$ grid in \cref{fig:simulated_res}.
For all three cases, Attention U-Net provides good depth localization but less accurate shape recovery, while fHybr yields precise shapes but limited depth resolution. 
This is illustrated by the first and fourth rows in \cref{fig:simulated_res} as well as the Maximum Intensity Projection (MIP) images on the two sides in  \cref{fig:simulated_3D}. 
WB-IPM combines these strengths to deliver clear boundaries, minimal artifacts, and accurate recovery of both shape and depth. In the most complex case (Case 3), it reconstructs ellipsoids at distinct depths with superior volumetric accuracy. 3D visualizations and MIP images for case 3 (\cref{fig:simulated_3D}) further demonstrate its superiority in boundary delineation and artifact suppression.

\begin{figure}[ht!]   
   \includegraphics[width = 0.9\textwidth]{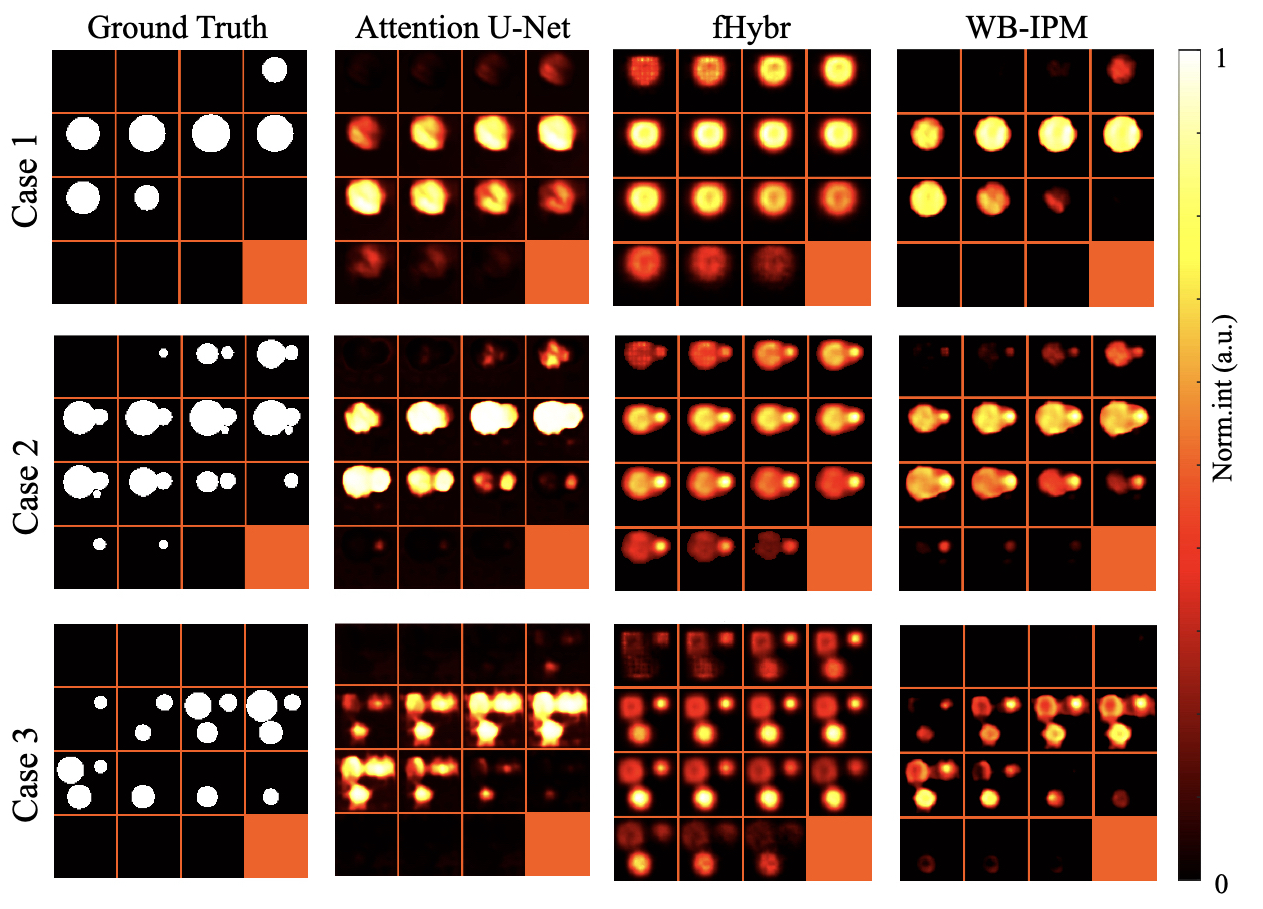}
   \centering
    \caption{Three simulated cases: reconstructions from Attention U-Net, fHybr, and WB-IPM, with slices along the $z$-axis.
    fHybr yields wrong results in the first and fourth rows.}
    \label{fig:simulated_res}
\end{figure}

\begin{figure}[ht!]   
   \includegraphics[width = 1\textwidth]{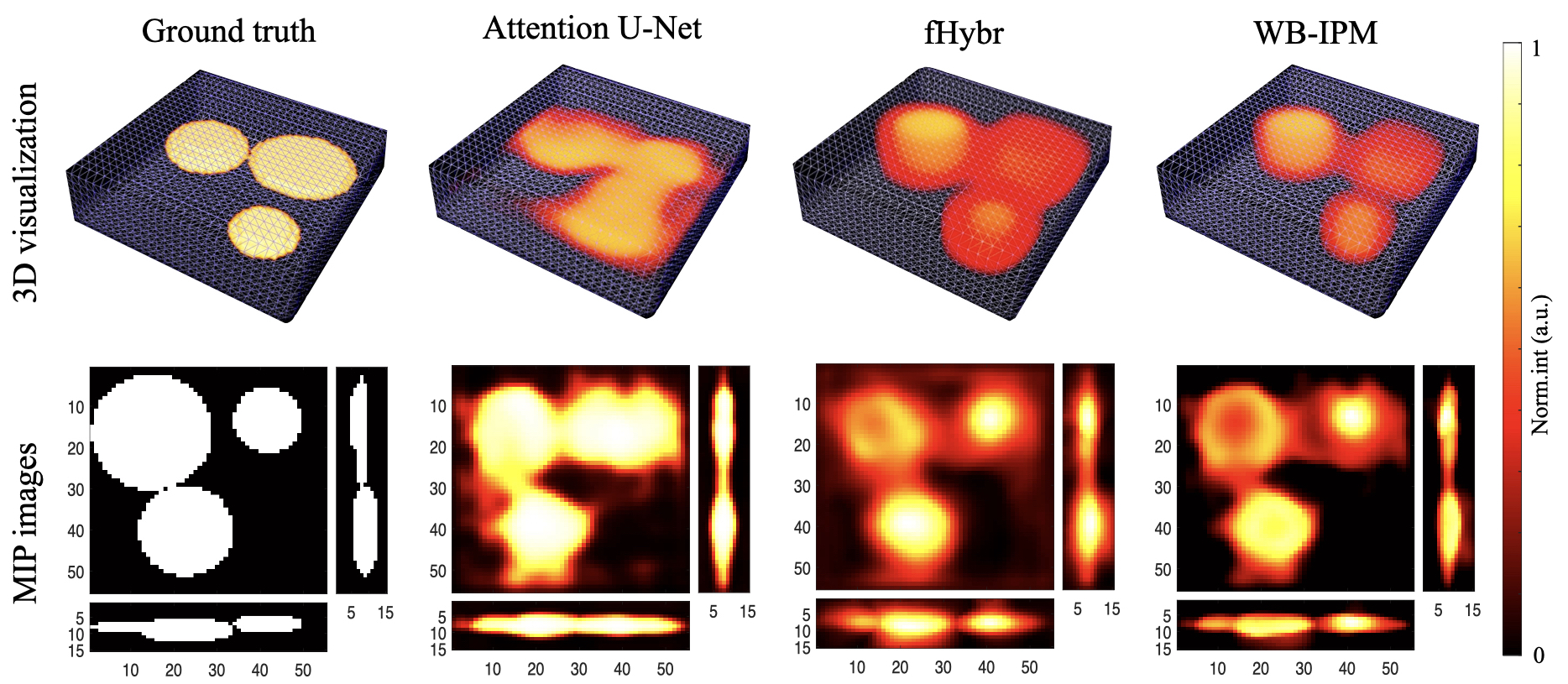}
   \centering
    \caption{3D visualizations (1st row) and MIP images (2nd row) for Case 3 of \cref{fig:simulated_res} reconstructed using  Attention U-Net, fHybr, and WB-IPM.
    }
    \label{fig:simulated_3D}
\end{figure}

\begin{figure}[ht!]   
   \includegraphics[width = 0.7\textwidth]{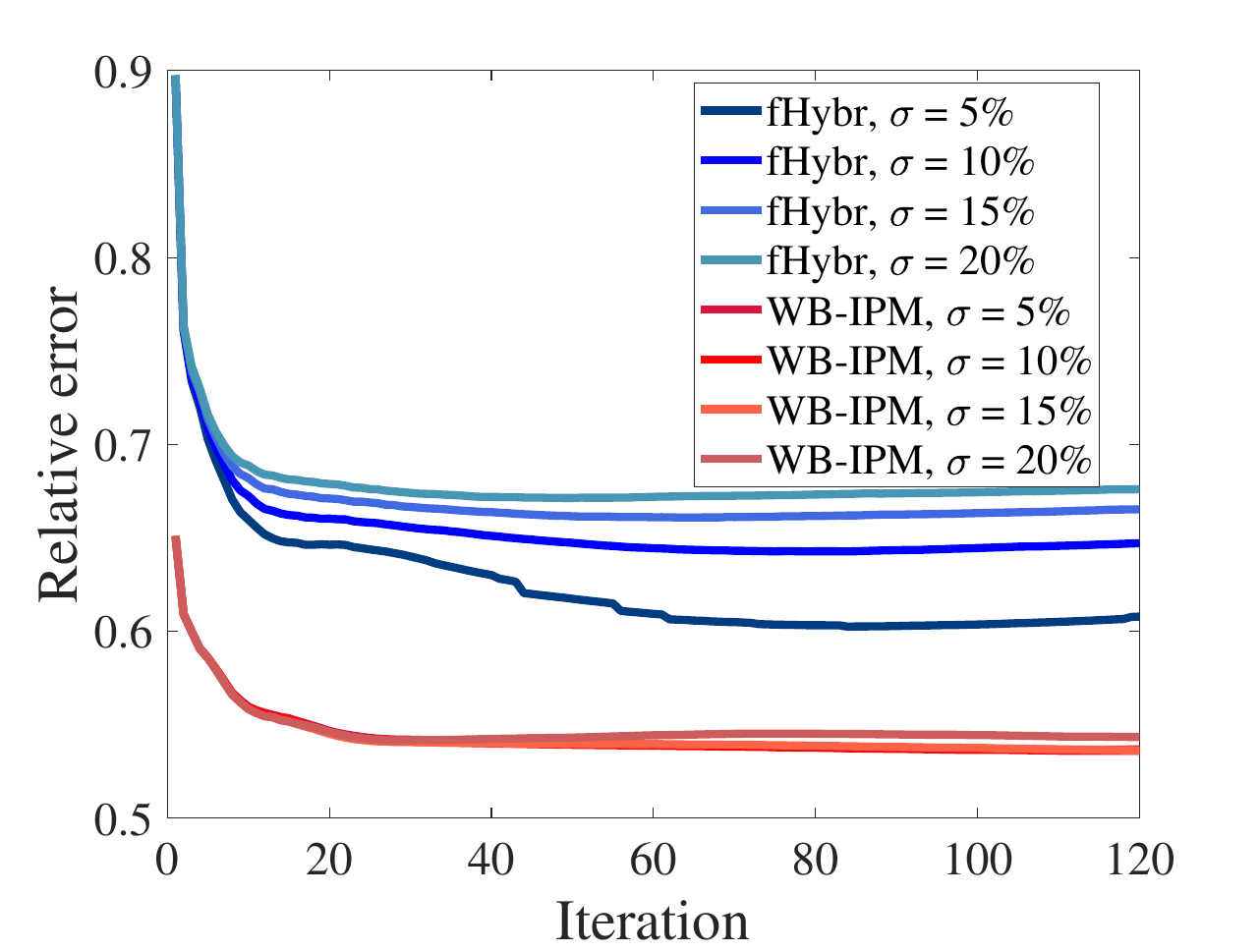}
   \centering
   \caption{Average relative errors across 50 simulated test datasets under noise levels of 5$\%$, 10$\%$, 15$\%$ and 20$\%$. Upon convergence, WB-IPM achieves relative error reductions of 9.54$\%$, 15.75$\%$, 17.95$\%$ and 18.96$\%$, compared to fHybr, with respect to noise levels from low to high. At 5\% noise, WB-IPM converges 2.5$\times$ faster (146.7 s for fHybr).
   $\sigma = \frac{\|\bbeta\|_2}{\|\bA \bx^*\|_2}$ denotes the noise level.}
    \label{fig:simu_error}
\end{figure}

\begin{figure}[ht!]   
   \includegraphics[width = 0.8\textwidth]{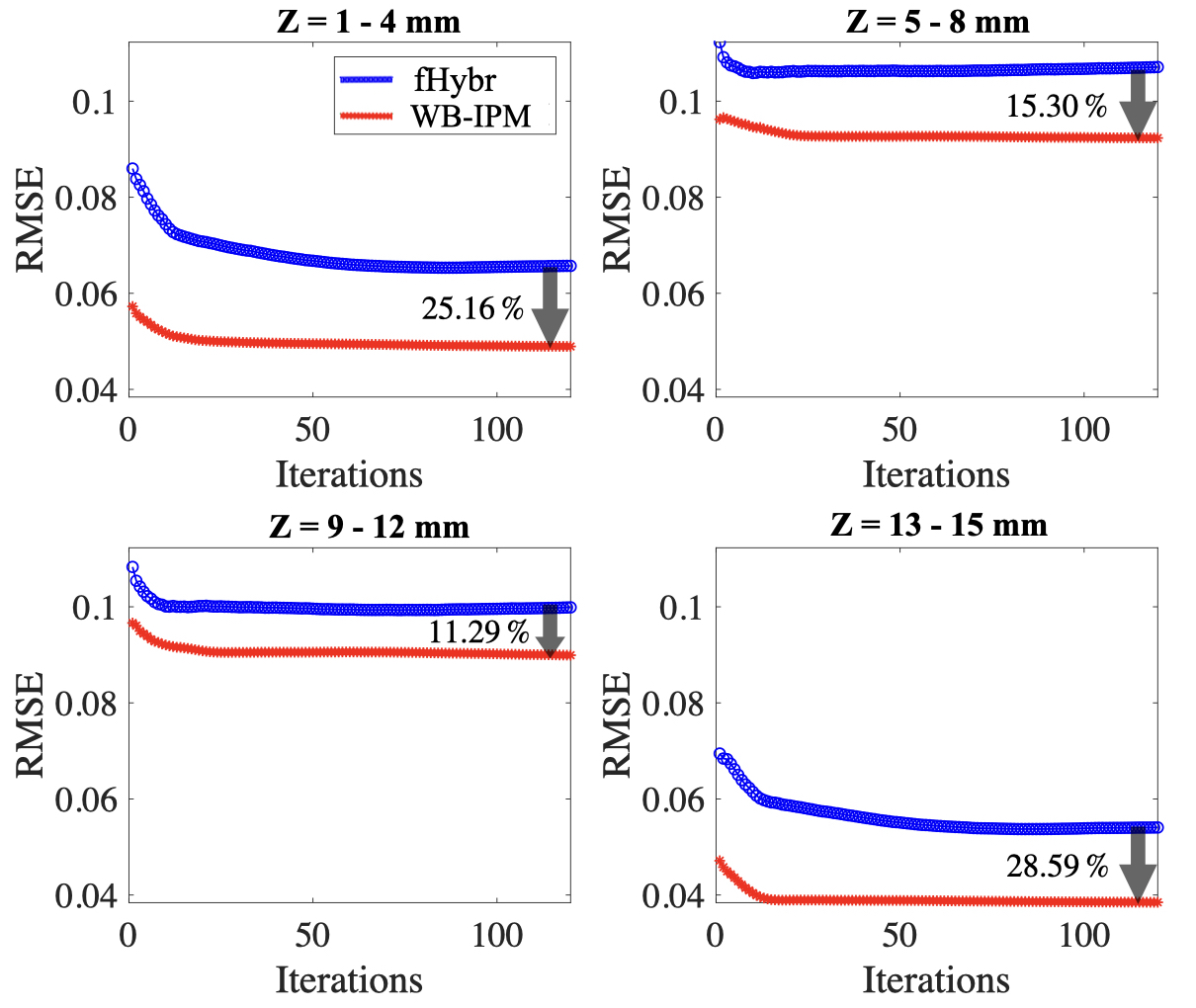}
   \centering
    \caption{Average RMSE across four $z$-axis sections for 50 simulated test cases with 10\% noise. After 120 iterations, WB-IPM reduces error compared to fHybr by 25.16\%, 15.30\%, 11.29\%, and 28.59\% along the $z$-axis. }
    \label{fig:z_resolution}
\end{figure}

 \cref{fig:simu_error} and \cref{fig:z_resolution} show that WB-IPM is highly robust to noise, with superior $z$-axis resolution and overall accuracy.
It consistently achieves the low relative error across noise levels (5–20\%), remaining below $0.53$ even at 20\% noise (\cref{fig:simu_error}), while fHybr suffers from significant degradation. Since the ground truth is nearly zero at the top and bottom slices along $z-$axis, we compute the average root mean squared error (RMSE) instead of relative errors across $50$ simulated test datasets and observe that WB-IPM outperforms fHybr by 25.16\% and 28.59\% in the boundary regions ($z=1$–4 mm and $z=13$–15 mm), resulting in clearer depth localization and finer structural recovery (\cref{fig:z_resolution}). The RMSE improvements in \cref{tab:error_results} further confirm these findings.
Remarkably, WB-IPM attaining higher accuracy and robustness even with only 20 iterations.

\begin{table}[ht!]
    \centering
   \caption[RMSE improvement (short)]{%
  \textbf{RMSE improvement of WB-IPM methods in the simulation study.}
  Average RMSE improvement at iterations $k = 20, 50, 120$ across
  four sections along the $z$-axis, under noise levels of 5\%, 10\%, 15\%, and 20\%. }
    \label{tab:error_results}
    \footnotesize{
    \begin{tabular}{cc|c|ccc}
        \toprule
        \multirow{2}{*}{\textbf{$z$-axis}} & \multirow{2}{*}{\textbf{Noise}} & \multirow{2}{*}{\textbf{fHybr}}  &  \multicolumn{3}{c}{\textbf{WB-IPM (\%)}} \\
        \cmidrule(lr){4-6} 
        \textbf{(mm)}& \textbf{Level}& \textbf{(RMSE)}& $k=20$ &  $k=50$ &  $k=120$ \\
        \midrule
        \multirow{4}{*}{$1-4$} & 5\%& 0.0581 
        & 14.01 &  14.11 &  14.20\\ 
        & 10\%&  0.0654  & 23.26&   24.25&   25.16\\
        & 15\%& 0.0686  & 26.27&   27.59&   28.09\\
        & 20\%& 0.0705 & 28.04&   29.18&  29.27 \\
        \midrule
        \multirow{4}{*}{$5-8$} & 5\%  &  0.1062  &12.25 &   12.43 &  12.47\\
        & 10\%& 0.1090  & 14.64&   14.98&   15.30\\
        & 15\%& 0.1102  & 15.58&  15.94&  16.35\\
        & 20\%& 0.1111  & 16.15&  16.64&  17.21 \\
        \midrule
        \multirow{4}{*}{$9-12$} & 5\%& 0.0980  & 6.78 &6.96 & 7.07 \\
        & 10\%& 0.1014 & 10.40&   10.66&  11.29\\
        & 15\%& 0.1028  & 11.72&  12.37&  12.96 \\
        & 20\%& 0.1042  & 12.62&  13.42&  14.16\\
        \midrule
        \multirow{4}{*}{$13-15$} & 5\%    & 0.0465 & 16.62 & 16.77&   16.79 \\
        & 10\%& 0.0539  & 27.74&   27.86&   28.59\\
        & 15\%& 0.0569  & 30.81&  30.92&  31.13\\
        & 20\%& 0.0585 &  31.84&   31.65&   31.44 \\

        \midrule
       \multirow{4}{*}{Overall} & 5\%&  0.0743 &  12.13& 12.48& 12.66  \\
        & 10\%&  0.0844  & 19.01&   19.44&   20.09\\
        & 15\%&  0.0865  & 21.09&   21.70&  22.13\\
        & 20\%&  0.0879 &  22.16&   22.72&   23.01 \\

        \bottomrule 
    \end{tabular}}
\end{table}

\subsection{Experimental results\label{sec:exp_results}}
To evaluate our method in practice, we conducted experiments with a silicone slab phantom designed to mimic biological tissue (\cref{fig:exp_error}, left).  The phantom was prepared with silicone mixed with TiO$_2$ and carbon black to reproduce scattering and absorption properties and included a peanut-shaped fluorescent component containing Cy5 dye at $0.0243$ µmol/ml. Measurements were acquired using a multifunctional FMT system \cite{gao2023multifunctional} in transmission mode, with a $10\times10$ laser grid on the bottom surface and a $55\times55$ detector grid on the top surface, consistent with the simulated setup.

\begin{figure}[ht!]   
   \includegraphics[width = 1\textwidth]{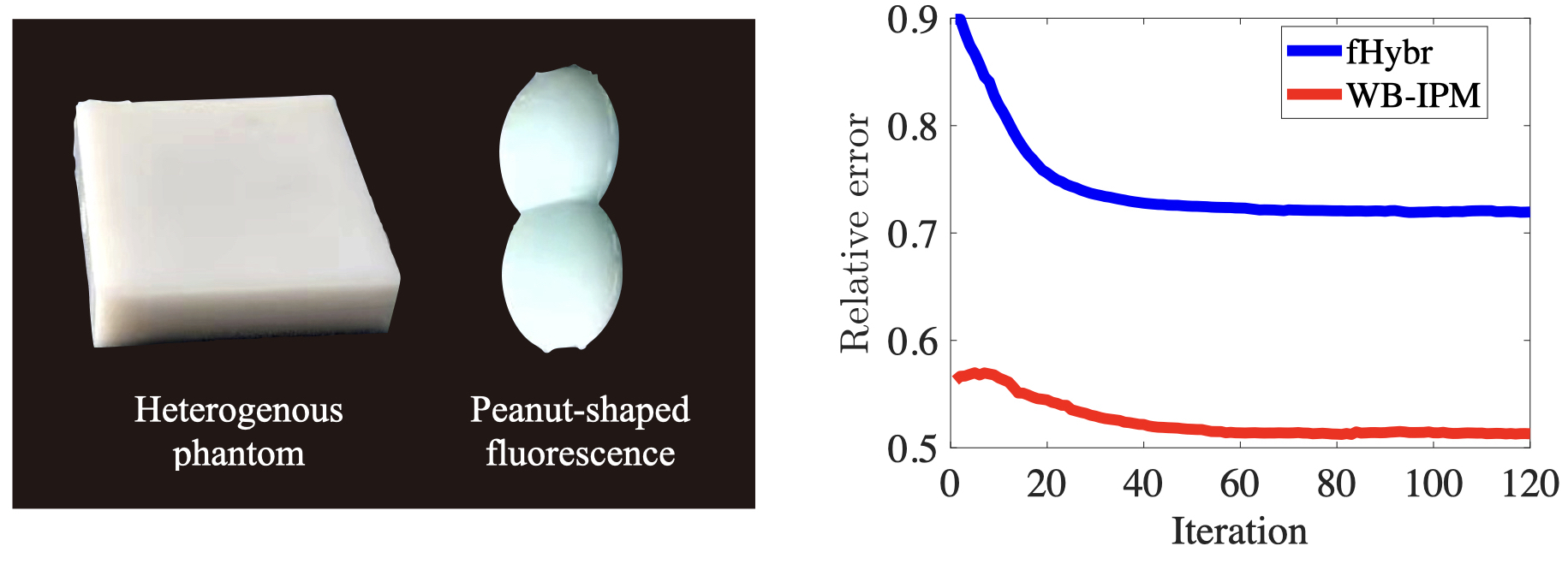}
   \centering
    \caption{Left: Silicone slab phantom with a central peanut-shaped fluorophore.
    Right: Relative errors for experimental study. WB-IPM consistently outperforms fHybr with 27.8\% lower error.}
    \label{fig:exp_error}
\end{figure}

\begin{figure}[ht!]   
   \includegraphics[width = 1\textwidth]{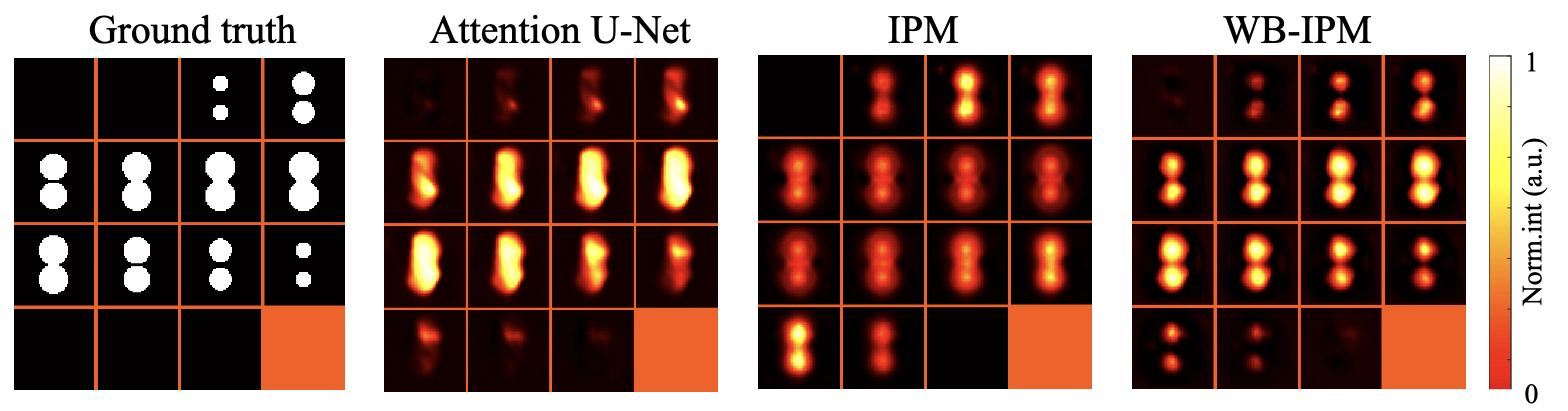}
   \centering
    \caption{Results of real silicone slab phantom, obtained using Attention U-Net, fHybr, and WB-IPM, with slices depicted along the $z$-axis.}
    \label{fig:exp_res}
\end{figure}

\cref{fig:exp_res} shows reconstructions from experimental measurements.
 Attention U-Net captures the general inclusion shape and $z$-axis localization but produces blurred boundaries and misses fine structural details, such as the connection between the two ellipsoids. 
fHybr recovers partial shape details and the connection structure but suffers from poor depth localization: the brightest regions shift toward the ends of the $z$-axis, and axial slices fail to reflect the true morphological transitions.
Similar to simulated case studies, WB-IPM yields reconstructions closest to the ground truth, with sharp boundaries, minimal artifacts, and substantially improved depth accuracy. The 3D visualizations and MIP images (\cref{fig:exp_3D}) indicate that only WB-IPM restores depth, shape, and edge details with high fidelity.

\begin{figure}[ht!]   
   \includegraphics[width = 0.8\textwidth]{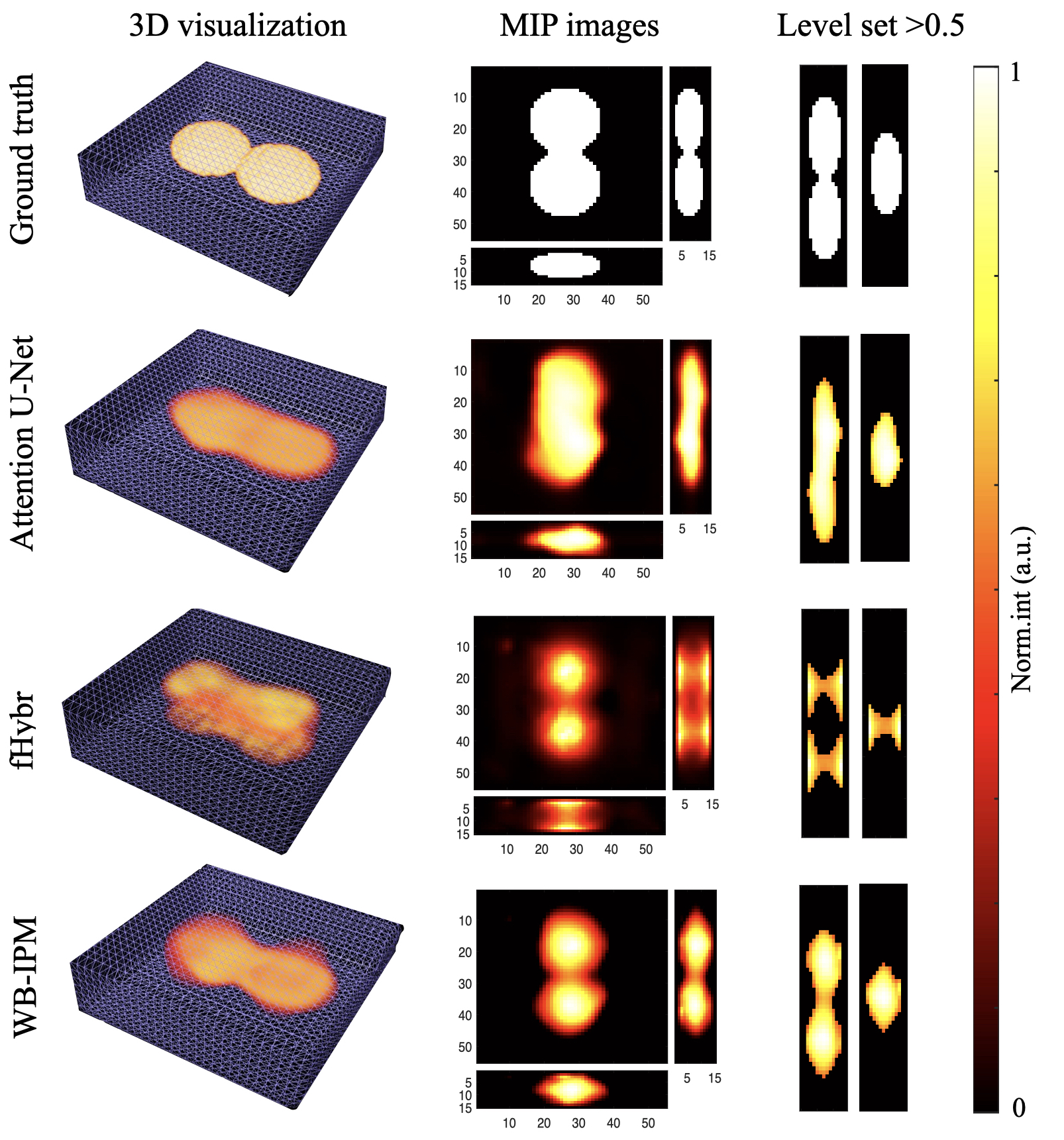}
   \centering
    \caption{ Comparative visualization of experimental reconstructions (\cref{fig:exp_res}) by various methods, including Attention U-Net, fHybr, and WB-IPM. Left: 3D renderings with FEM mesh; Center: MIP images; Right: side-view level-set images of the MIP at threshold 0.5.}
    \label{fig:exp_3D}
\end{figure}

For  quantitative evaluation, we track the relative error across iterations (\cref{fig:exp_error}, right) and report RMSE improvements along the $z$-axis in \cref{tab:exp_error_results}. The Attention U-Net already achieves a 27.8\% lower error than the final fHybr result, underscoring its stronger $z$-axis representation (\cref{fig:exp_res}). Building on this initialization, WB-IPM further refines the solution, reaching a final relative error of about 0.51. The results in \cref{tab:exp_error_results} show it yields consistent error reductions across all $z$-sections and an overall 28.01\% RMSE improvement over fHybr. These results indicate that our method remains reliable in practical problem settings.


\begin{table}[ht!]
    \centering
   \caption[RMSE improvement (short)]{%
  \textbf{RMSE improvement of WB-IPM methods in the experimental study.}
    Average RMSE improvement at iterations $k = 20, 50, 120$ across four sections along the $z$-axis. The noise is unknown is this case. }
    \label{tab:exp_error_results}
    \vspace{1em}
    \small{
    \begin{tabular}{c|c|ccc}
        \toprule
        \multirow{2}{*}{\textbf{$z$-axis}} & \multirow{2}{*}{\textbf{fHybr}} & \multicolumn{3}{c}{\textbf{WB-IPM (\%)}} \\
        \cmidrule(lr){3-5} 
        \textbf{(mm)}& \textbf{(RMSE)}   & $k=20$ &  $k=50$ &  $k=120$ \\
        \midrule
       $1-4$  &    0.0752 & 0.67 &   1.72  &  2.75 \\
      $5-8$  & 0.1807  &   27.07 &   31.09   & 31.69\\
      $9-12$    &  0.1139  &   31.24  &  33.94  &  34.87\\
       $13-15$   &  0.0406 &   33.34  &  35.98  &  36.33 \\
        \midrule
       \multirow{1}{*}{Overall}  &   0.1067  & 24.21 &   27.36 &   28.01 \\
                 
        \bottomrule 
    \end{tabular}}
\end{table}

\section{Conclusion}
\label{sec:conclusion}
We have proposed WB-IPM for large-scale inverse problems, and illustrated its potential on FMT.
By integrating Attention U-Net predictions as a basis into the alternative solver in two subspaces and adopting the AFGK process to efficiently solve the subproblems, WB-IPM combines the strengths of learning and iteration: the network captures depth information, and the iterative solver refines it with stability and accuracy. Our analysis further establishes error bounds that depend only on the alignment between the true solution and the network output, apart from noise and regularization parameters.   Both simulation and experimental studies confirm that WB-IPM achieves more accurate, robust, and efficient reconstructions than either pure network predictions or standard iterative solvers, particularly in recovering depth information. 
 Remarkably, our approach allows training under a weaker loss for greater efficiency, without sacrificing the final accuracy after iteration.
 In practice,  WB-IPM also shows strong robustness to noise.

\appendix
\section{Efficient QR update of $\boldsymbol{Z}_k$\label{sec:eff-QR}}
Suppose we have already computed the thin QR factorization
\[
\bZ_k = \bQ_{Z,k}\bR_{Z,k},
\]
where $\bQ_{Z,k} \in \mathbb{R}^{N\times k}$ has orthonormal columns, and $\bR_{Z,k} \in \mathbb{R}^{k\times k}$ is upper triangular.
When a new column $\bz_{k+1} \in \mathbb{R}^N$ arrives from the $(k+1)$th iteration, we form the augmented matrix
\[
\bZ_{k+1} \;=\; \bigl[\bZ_k \;\;\; \bz_{k+1}\bigr].
\]
To update the factorization efficiently, we seek
\begin{equation}
\label{eq:Wkplus1}
\mathbf{Z}_{k+1}
\;=\;
\underbrace{\bigl[\,
\mathbf{Q}_{Z,k},\;
\tfrac{\bigl(\mathbf{I} - \mathbf{Q}_{Z,k}\mathbf{Q}_{Z,k}^{T}\bigr)\,\mathbf{z}_{k+1}}{r_{k+1,k+1}}
\bigr]}_{\displaystyle \mathbf{Q}_{Z,k+1}}
\underbrace{\begin{pmatrix}
\mathbf{R}_{Z,k} & \mathbf{Q}_{Z,k}^{T}\,\mathbf{z}_{k+1}\\[6pt]
\mathbf{0}^{T} & r_{k+1,k+1}
\end{pmatrix}}_{\displaystyle \mathbf{R}_{Z,k+1}},
\end{equation}
where $ r_{k+1,k+1} = \norm{\bigl(\mathbf{I} - \mathbf{Q}_{Z,k}\mathbf{Q}_{Z,k}^{T}\bigr)\,\mathbf{z}_{k+1}}_2$, $\mathbf{Q}_{Z,k+1}$ remains orthogonal  by appending the normalized residual and the $\mathbf{R}_{Z,k+1}$ extends $\mathbf{R}_{Z,k}$ while preserving its upper triangular structure. This procedure updates the QR factorization in $\mathcal{O}(Nk)$ operations, avoiding the need for a full QR decomposition at each iteration. To enhance numerical stability, one may employ modified Gram--Schmidt, a second orthogonalization pass, or Householder-based updates.

\bibliographystyle{siam}
\bibliography{references}

\end{document}